\documentclass[reqno, 11pt]{amsart}
\usepackage[percent]{overpic}
\usepackage{calc,graphicx,amsfonts,amsthm,amscd,epsfig,psfrag,amsmath,amssymb,dsfont,enumerate, mathrsfs,paralist}

\usepackage{subfig} 
\usepackage{pdfsync}
\usepackage{stmaryrd} 
\usepackage{marginnote}
\usepackage[top=3.5cm, bottom=3.5cm, outer=3.5cm, inner=3.5cm, heightrounded, marginparwidth=2.2cm, marginparsep=0.5cm]{geometry}

\usepackage[american]{babel}
\usepackage{amsmath}
\usepackage{dsfont,mathtools,amssymb}
\usepackage{color}

\usepackage[showlabels,sections,floats,textmath,displaymath]{}

\usepackage[showlabels,sections,floats,textmath,displaymath]{}

\reversemarginpar
\newlength\fullwidth
\numberwithin{equation}{section}

\DeclareMathSymbol{\leqslant}{\mathalpha}{AMSa}{"36} 
\DeclareMathSymbol{\geqslant}{\mathalpha}{AMSa}{"3E} 
\DeclareMathSymbol{\eset}{\mathalpha}{AMSb}{"3F}     
\renewcommand{\leq}{\;\leqslant\;}                   
\renewcommand{\geq}{\;\geqslant\;}                   

\newcommand{\rb}{{\rm b}}
\newcommand{\rw}{{\rm w}}

\def\1{\ifmmode {1\hskip -3pt \rm{I}} \else {\hbox {$1\hskip -3pt \rm{I}$}}\fi}

\newcommand{\grad}{\nabla}

\newcommand{\z}{\zeta}

\renewcommand{\epsilon}{\varepsilon}

\newtheorem{theorem}{Theorem}[section]

\newtheorem{lemma}[theorem]{Lemma}
\newtheorem{proposition}[theorem]{Proposition}
\newtheorem{corollary}[theorem]{Corollary}

\newtheorem{remark}[theorem]{Remark}

\newtheorem{definition}[theorem]{Definition}

  \let\z=\zeta

\renewcommand{\le}{\leq}
\renewcommand{\ge}{\geq}

\date{}
\title[A $(2+1)$-dimensional growth process]{A $(2+1)$-dimensional growth process
  with explicit stationary measures}
\author{Fabio Lucio Toninelli}
\address{Universit\'e de Lyon, CNRS and Institut Camille Jordan, Universit\'e Lyon 1,
    43 bd
 du 11 novembre 1918, 69622 Villeurbanne, France}
\email{toninelli@math.univ-lyon1.fr}
\begin{document}

\begin{abstract}
We introduce a class of   $(2+1)$-dimensional stochastic growth processes,
that can be seen as irreversible random  dynamics of
discrete interfaces. ``Irreversible'' means that the interface has an average non-zero
drift.  Interface configurations correspond to
height functions of dimer coverings of the infinite hexagonal or square
lattice.  The model can also
be viewed as an interacting driven particle system and in the totally
asymmetric case  the dynamics corresponds to an infinite
collection of mutually interacting Hammersley processes.

When the dynamical asymmetry parameter $(p-q)$ equals zero, the
infinite-volume Gibbs measures $\pi_\rho$ (with given slope $\rho$) are stationary and
reversible. When $p\ne q$, $\pi_\rho$ are not reversible any more but,
remarkably, they are  still stationary. In such stationary states, we
find that the average
height function at any given point $x$ grows linearly with time $t$ with a
non-zero speed: $\mathbb E Q_x(t):=\mathbb E(h_x(t)-h_x(0))= V(\rho) t$ while the typical
fluctuations of $Q_x(t)$ are smaller than any power of $t$ as $t\to\infty$.

In the totally asymmetric case of  $p=0,q=1$ and on the hexagonal lattice, the dynamics coincides
with the ``anisotropic KPZ growth model'' introduced by A. Borodin and
P. L. Ferrari in \cite{BF1,BF2}. For
a suitably chosen, ``integrable'',  initial condition (that is very far from the
stationary state), they were able to determine the
hydrodynamic limit and a CLT for interface fluctuations on scale
$\sqrt{\log t}$, exploiting the fact
that in that case certain space-time height correlations can be
computed exactly. 
In the same setting they proved that, asymptotically for $t\to\infty$, the local
statistics of height fluctuations tends to that of a Gibbs state
(which led to the prediction that Gibbs states should be stationary).
\end{abstract}
\thanks{This work was partially supported by the Marie Curie
  IEF Action ``DMCP- Dimers, Markov chains and Critical Phenomena'',
  grant agreement n.  621894}
\maketitle

\section{Introduction}

To motivate the object of our study, let us start with a
well-known $(1+1)$-dimensional growth process. At all times $t$, the
configuration is an integer-valued height function $x\in \mathbb
Z\mapsto h_x(t)\in\mathbb Z$ with space increments
$h_x-h_{x-1}=\pm1$, see Fig. \ref{fig:aseppo}. Local minima turn to local maxima
with rate $p$ (this corresponds to deposition of elementary squares)
and local maxima to local minima with rate $q$ (evaporation of
elementary squares). If positive interface
gradients are identified with ``particles'' and negative gradients with
``holes'', this process is
equivalent to the one-dimensional Asymmetric Simple Exclusion
process (ASEP). 
\begin{figure}[h]
  \includegraphics[width=7cm]{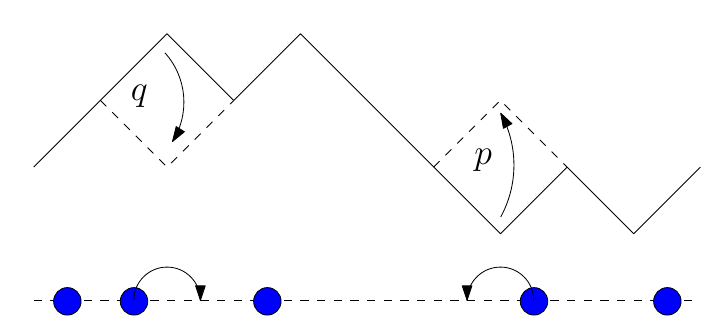}
\caption{The ASEP: squares are deposed (i.e. minima turn to maxima)
  with rate $p$ and evaporate (i.e. maxima turn to minima) with rate
  $q$. In the particle interpretation, particles jump to the right/left with
rate $q/p$ and cannot occupy the same site.}\label{fig:aseppo}
\end{figure}

The study of this and similar stochastic growth processes in dimension $(1+1)$ witnessed a
spectacular progress recently, especially in relation with the
so-called KPZ equation, cf. e.g. \cite{Quastel,FSp,CRM} for  recent reviews. Some
of the basic questions that were solved for certain models include the
identification of the translation-invariant stationary states (for ASEP, these are simply
the combinations of Bernoulli measures for any intensity
$\rho\in[0,1]$), the determination of the dynamic scaling exponents characterising the
 space-time correlation structure of height fluctuations, the study of
the limit rescaled fluctuation process and its dependence on the type
of initial condition. 
The same KPZ scaling relations appear also in the context of $(1+1)$-dimensional
directed polymers in random environment, last passage percolation and
random matrix theory, just to mention a few instances \cite{Quastel,FSp,CRM}.

 On the other hand, for $(d+1)$-dimensional
stochastic growth models, $d\ge 2$, the situation is much more rudimentary and
mathematical results (see notably \cite{PS,BF1}) are rare.
In this work we introduce a $(2+1)$-dimensional stochastic growth
process, for which we study the stationary measures and the
corresponding large-time behavior of height fluctuations.  The
two-dimensional interfaces entering the definition of our process are
discrete (i.e. heights are integer-valued) and are given by the height
function associated to dimer coverings (perfect matchings) of either
the infinite hexagonal or infinite square lattice \cite{Klecturenotes}.  Height functions
corresponding to dimer coverings of bipartite planar graphs, or to the
associated tilings of the plane, are classical examples of discrete
two-dimensional interfaces.  For instance, dimer coverings of the
hexagonal lattice (i.e. tilings of the plane by lozenges of three
different orientations) correspond to discrete monotone surfaces
obtained by stacking unit cubes, see Figure
\ref{fig:cubi}. ``Monotone'' means that if we let $h_{x,y}$ denote the
height w.r.t. the horizontal plane of the vertical column of cubes with horizontal coordinates
$(x,y)$, then $h_{x,y}\ge \max(h_{x+1,y},h_{x,y+1})$. In a sense,
discrete monotone height functions are the most natural $(2+1)$-dimensional 
analogue of the $(1+1)$-dimensional height functions appearing in the one-dimensional
ASEP.

Given a
density vector $\rho=(\rho_1,\rho_2,\rho_3)\in \mathbb R_+^3$ with
$\rho_1+\rho_2+\rho_3=1$, there exists \cite{KOS} a unique infinite-volume
translation-invariant ergodic Gibbs measure $\pi_\rho$ such that
\begin{itemize}
\item the
three types of lozenges have densities $\rho_i,i=1,2,3$ and
\item conditioned
on the tiling configuration outside a finite region $\Lambda$ of the plane,
$\pi_\rho$ describes a uniformly random tiling of $\Lambda$.
\end{itemize}
The measures $\pi_\rho$ have an explicit determinantal structure that
will play a role in this work and that is recalled in Section \ref{sec:altezza}.
 
\begin{figure}[h]
  \includegraphics[width=6cm]{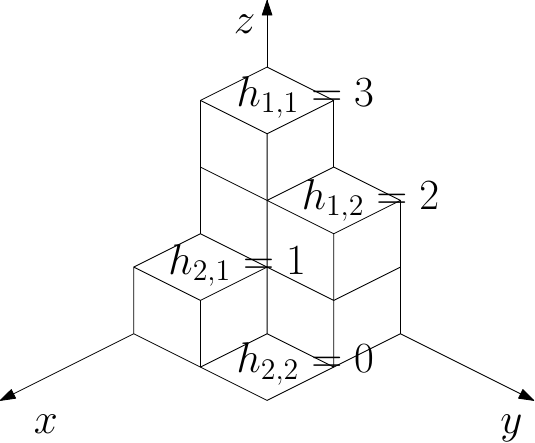}
\caption{A lozenge tiling of a portion of the plane, or equivalently a monotone
stacking of unit cubes}\label{fig:cubi}
\end{figure}

To model a growth process, we want to introduce a Markov evolution
which is \emph{asymmetric} or \emph{irreversible}, in the sense that
the interface has a net drift, proportional to an asymmetry parameter
$p-q$.  Moreover, as discussed in Section \ref{sec:kpz} below, in
order that its fluctuations can be at least heuristically described by
a $(2+1)$-dimensional KPZ-type equation, the average interface speed
should be a non-linear function of the interface slope.  The most
natural $(2+1)$-dimensional generalization of the ASEP described above
(but which is \emph{not} the one we will study here) would be the
following. Let
\begin{equation}
  \begin{aligned}
\Delta^+_{x,y}:&=\min(h_{x-1,y},h_{x,y-1})-h_{x,y}\ge 0\\
\Delta^-_{x,y}:&=h_{x,y}-\max(h_{x+1,y},h_{x,y+1})\ge0,    
  \end{aligned}
\end{equation}
and observe that $\Delta^+_{x,y}$ (resp. $\Delta^-_{x,y}$) is the maximal number of cubes we
can add to (resp. remove from) column $(x,y)$ while respecting the
condition $h_{x',y'}\ge \max(h_{x'+1,y'},h_{x',y'+1})$ for every $(x',y')$.
For every column $(x,y)$,
we add a single cube with rate $p$ if $\Delta^+_{x,y}>0$ and remove a
single cube with rate $q$ if $\Delta^-_{x,y}>0$.  In words, single
elementary cubes are deposed (Fig. \ref{fig:flip} top) with rate $p$
and removed (Fig. \ref{fig:flip} bottom) with rate $q$ (compare with
Fig. \ref{fig:aseppo}).  We refer to this as the ``single-flip
dynamics''. If $p=q$ there is no drift and the infinite-volume Gibbs
measures $\pi_\rho$ \cite{KOS} are stationary and reversible.  If
instead $p\ne q$, the stationary states are not known, but they appear
to be definitely very different from the equilibrium Gibbs measures
\cite{FT1,FT2,NT}.  This process has been studied numerically and one
finds that typical interface fluctuations grow with time like
$t^\beta$, with $\beta\simeq 0.24...$ \cite{FT1,FT2}.  This is in
sharp contrast with the ASEP, where the Bernoulli measures are
stationary, irrespective of $p$ being equal or different from $q$.  In
the language of Section \ref{sec:kpz}, the two-dimensional single-flip
growth process is believed to belong to the so-called \emph{isotropic}
$(2+1)$-dimensional KPZ class when $q\ne p$.  Unfortunately, the single-flip process
is very hard to analyze mathematically and very little is known
rigorously.

In this work we study, instead of the single-flip dynamics, a
different $(2+1)$-dimensional irreversible growth process, that we
call ``bead dynamics'' for reasons that will be clear later (in the
hexagonal lattice case, ``beads'' or ``particles'' correspond to
horizontal lozenges as in Fig. \ref{fig:cubi}).  As discussed in
Section \ref{sec:kpz}, the bead dynamics belongs (in contrast with the
single-flip dynamics) to the so-called \emph{anisotropic}
$(2+1)$-dimensional KPZ class when $q\ne p$.  Updates of the dynamics
consist in adding or removing a random number $\ge 1$ of cubes at some
column $(x,y)$, in the following way (see Section \ref{sec:dyna} for a
precise definition and Section \ref{sec:estensionedomini} for the
analogous construction on the square lattice). For every column
$(x,y)$, we assign
\begin{itemize}
\item rate $p$ to the update $h_{x,y}\to h_{x,y}+i$
for every $i=1,\dots,\Delta^+_{x,y}$ (deposition of $i$ cubes to
  column $(x,y)$);
\item rate $q$ to the update $h_{x,y}\to h_{x,y}-i$ for every
  $i=1,\dots,\Delta^-_{x,y}$ (removal of  $i$ cubes
  from column $(x,y)$).
\end{itemize}
If $p=q$ again there is no drift and the measures $\pi_\rho$
\cite{KOS} are stationary and
reversible.  Somewhat surprisingly, $\pi_\rho$
turns out to be stationary (but not reversible!) for any density vector
$\rho$ and for
any value of $p-q$. This is the content of our first result, Theorem
\ref{th:losanghe}. The same then clearly holds also if we add to the generator of the
bead dynamics the generator of another process w.r.t. which $\pi_\rho$
is reversible. The measures $\pi_\rho$ and their convex
combinations are the only stationary measures that can be obtained as
$L\to\infty$ limits of
stationary measures for the bead dynamics periodized on the torus of side
$L$. In principle our result does not exclude the existence of other
stationary measures that cannot be obtained this way; there might exist
for instance analogs of the so-called ``blocking measures'' of
one-dimensional asymmetric exclusion processes \cite{FLS,BrMo}.

We emphasize that it is a non-trivial fact that equilibrium Gibbs
measures should remain stationary in presence of dynamical
irreversibility. As we mentioned above, this is false for instance for
the single-flip dynamics.
Typically, one expects that a Gibbs measure of a reversible dynamics
remains stationary after 
introduction of a drift only when the reversible dynamics satisfies a
so-called ``gradient condition'' \cite{Spohn,KLS,Bertini}. As we
discuss in Section \ref{sec:gradiente}, for the symmetric 
dynamics with $p=q$ one can indeed identify a certain
``gradient condition'' that might help explain why
Theorem \ref{th:losanghe} holds.

\smallskip

It is important to emphasize that stationarity of the Gibbs measures
means that, if the process is started from the distribution
$\pi_\rho$,  the law of interface gradients is
time-invariant. However,  
overall the height function has  a time-dependent
random shift $h_{x_0}(t)-h_{x_0}(0)$ where, say, $x_0$ is the origin of the
plane.  On average $h_{x_0}(t)-h_{x_0}(0)$ grows like $(p-q) t\,V$ for some
non-zero and slope-dependent $V$ but the amplitude of its fluctuations cannot be
deduced immediately from the stationary gradient measure $\pi_\rho$. 
 Our second
result, Theorem \ref{th:vvar}, says that the typical fluctuations of $h_x(t)-h_x(0)$ grow
slower than any power of $t$. Under a certain (technical) restriction
on the interface slope, we can actually prove that fluctuations are at
most of order $\sqrt{\log t}$, which we believe
to be the optimal order of magnitude. Recall that, in sharp contrast, for the
single-flip dynamics  fluctuations were observed numerically
\cite{FT1,FT2}  to
grow like a non-trivial power of $t$.

A word about Theorem \ref{th:losanghe} (stationarity of $\pi_\rho$).
Checking stationarity is easy for the process periodized on the torus
of size $L$, see Section \ref{sec:sultoro}. The extension to the
infinite lattice is, however, non-trivial. One may expect that, when
$L$ is large, on local scales and for finite times the system does not
feel the periodic boundary conditions and therefore locally the
dynamics on the torus and on the infinite lattice could be coupled
with high probability. The situation is however more subtle: while on
the torus the process is always well-defined, in the infinite systems
one can easily construct initial configurations such that, for
instance, beads (horizontal lozenges) escape instantaneously to
infinity. This is due to the fact that we allow for an unbounded
amount of cubes to be deposed/removed at a time, since
$\Delta^\pm_{x,y}$ is not bounded. In order for the coupling to work,
one needs to prove that for typical initial conditions and with high
probability, the random variables $\Delta^\pm_{x,y}$ remain
sufficiently tight in time during the out-of-equilibrium evolution.
An important ingredient in overcoming these difficulties is the work
\cite{Seppa} by Sepp\"al\"ainen on the one-dimensional Hammersley
process \cite{AD,Seppa,FM}.  In fact, viewing beads as particles, the
bead dynamics can be seen as a two-dimensional generalization of the
Hammersley process, or more precisely an infinite collection of
interacting Hammersley processes, see Fig. \ref{fig:biglie2} (a
different two-dimensional generalization of the Hammersley process was
introduced by Sepp\"al\"ainen in \cite{Seppa2}: in that case a full
hydrodynamic limit was obtained, but the stationary measures and the size of height fluctuations remain
unknown). As a side remark, the single-flip dynamics can be instead
visualized in a natural way as an infinite collection of mutually
interacting one-dimensional ASEPs, see caption of
Fig. \ref{fig:biglie2}.

As
we explain in some more detail in Section \ref{sec:intfluctu}, in the
totally asymmetric case $p=0,q=1$ and on the
hexagonal lattice, the bead dynamics is the same as the interacting
driven particle system introduced by
A. Borodin and P. L. Ferrari in \cite{BF1,BF2}. In \cite{BF1}, for a
specific, deterministic initial condition, the hydrodynamic limit and
the convergence of height fluctuations on scale $\sqrt{\log  t}$ to a Gaussian field
were obtained. For such initial condition, the above-mentioned problem
of proving that the dynamics is well-posed does not arise, simply
because each bead has a deterministic, time-independent maximal
position it can possibly reach, and therefore cannot escape to
infinity. 
As we mention in Section \ref{sec:intfluctu}, on the basis of
\cite[Prop. 3.2]{BF1} it was natural to conjecture our Theorem \ref{th:losanghe}.

\subsection{Isotropic and anisotropic KPZ classes}
\label{sec:kpz}
In order to predict whether the fluctuations of a $(2+1)$-dimensional
growth process should be described by a KPZ-type equation, one should
look at the Hessian of  $V$, the average interface velocity considered as a
function of the interface slope. Indeed, the evolution of the
fluctuations $h$ in the stationary state of slope $\grad \phi$ should
be governed on large space-time scales by a stochastic PDE of the type
\begin{eqnarray}
  \partial_t h=\nu \Delta h+Q(\partial_x h,\partial_y h)+\text{white noise},
\end{eqnarray}
with $\nu$ a diffusion coefficient and $Q(\cdot,\cdot)$ a quadratic form whose
corresponding symmetric $2\times2$ matrix is proportional to the Hessian of $V$
at $\nabla \phi$.  (At present, it is not known how to regularize such
equation in order to make it mathematically well-defined, as was done
recently for its one-dimensional analog \cite{Hairer}).

The
growth model is said to belong to the ``anisotropic
$(2+1)$-dimensional KPZ class'' 
when the two eigenvalues of the quadratic form $Q$ have  opposite sign,  and to the ``isotropic
$(2+1)$-dimensional KPZ class'' when they have the same sign.  As
discussed in \cite{BF2}, the bead dynamics belongs to the anisotropic
class (the eigenvalues can be computed explicitly from formula
\eqref{eq:V} below for $V$).

Models in the anisotropic class are in a sense easier than those in
the isotropic class.  Indeed, in the former case it was predicted by
Wolf \cite{Wolf} that the non-linearity $Q$ is irrelevant as far as
the large-time behavior of the interface roughness is concerned, i.e. the
fluctuations of $h_x(t)-h_x(0)$ should be of the same order
$\sqrt{\log t}$ as for the linear Edwards-Wilkinson equation
\cite{EW}, where $Q$ is set to zero.  Theorem \ref{th:vvar} and
Eq. \eqref{eq:varianzaQ1} confirm this prediction, for the bead
model. Apart from the bead dynamics we study here, there are a few other
$(2+1)$-dimensional stochastic growth model models known to be in the
anisotropic KPZ class, and all of them are exactly solvable in some
sense. In this respect, let us mention the model introduced by
Pr\"ahofer and Spohn in \cite{PS}, for which height fluctuations are also
known to grow like $\sqrt{\log t}$.  See also \cite[Sec. 3.3]{BBO} for
growth models in the same universality class: it would be
interesting to see whether our result extend to these processes.

The situation is very different for models in the isotropic KPZ class. In
this case there are, to our knowledge, no exactly solvable models and
only numerical simulations are available (see \cite{HHK} for an overview). The non-linearity $Q$ is
expected to be relevant and to produce a non-trivial dynamical height
fluctuation exponent. In particular, while neither the interface velocity $V$ nor the stationary states of the
$(2+1)$-dimensional single-flip dynamics can be computed explicitly, the model is widely believed to  belong to 
the isotropic KPZ class and, as mentioned above,
the dynamical fluctuation exponent is numerically estimated to
$\beta\simeq 0.24..$ \cite{FT1,FT2}.

\section{Irreversible lozenge dynamics and stationarity of Gibbs states}
\label{sec:goal}

\subsection{Configuration space}
\label{sec:definizioni}
The Markov process we are interested in lives on $\Omega_{\mathcal H}$, the set of dimer coverings (perfect
matchings) of the hexagonal lattice $\mathcal H$, or equivalently the set of lozenge tilings of the whole
plane. See Figure \ref{fig:losanghe}. 
\begin{figure}[h]
  \includegraphics[width=10cm]{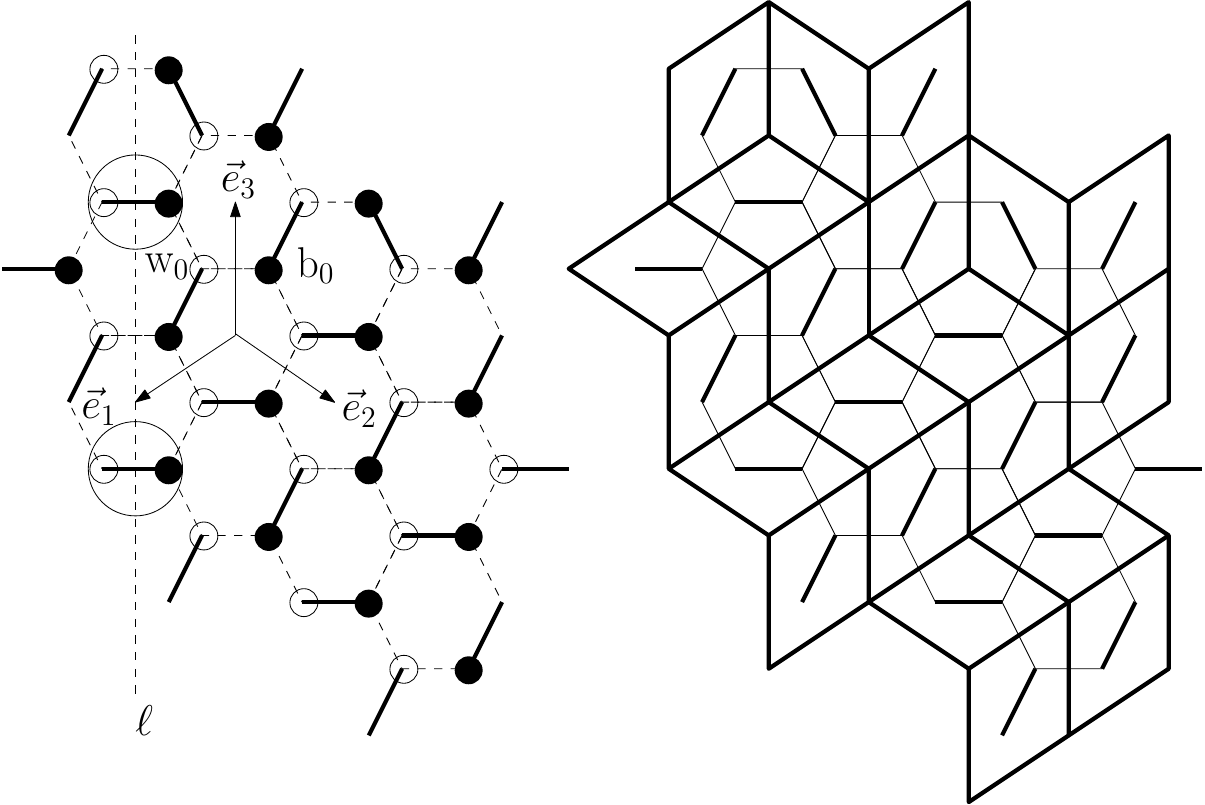}
\caption{A portion of dimer covering of $\mathcal H$  and the corresponding lozenge
  picture. The two beads in column $\ell$ are encircled and the
  vertices $\rb_{0},\rw_{0}$ are marked.}
\label{fig:losanghe}
\end{figure}
The ``elementary moves'' of the dynamics consist
in rotating by an angle $\pi/3$ three dimers around a hexagonal face, see
Figure \ref{fig:flip}. 
\begin{figure}[h]
  \includegraphics[width=4cm]{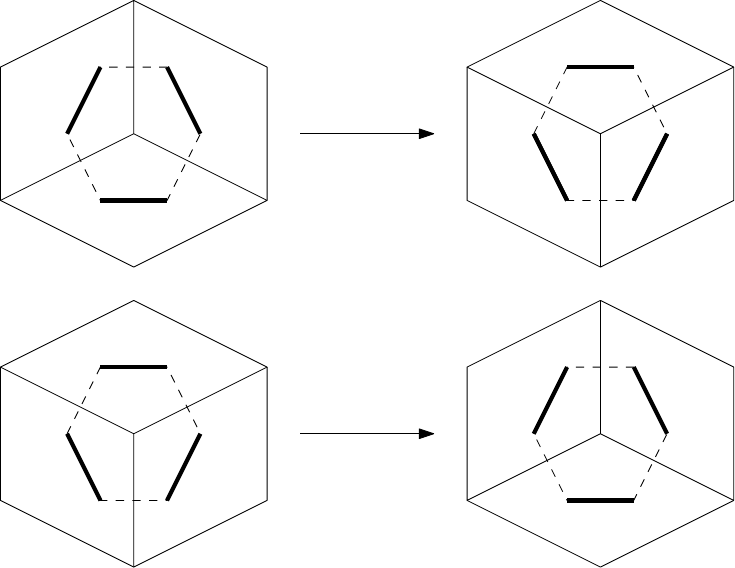}
\caption{The two  elementary moves.}
\label{fig:flip}
\end{figure}
In this move, a horizontal dimer moves up or
down a distance $1$. The generic move of the dynamics
(defined precisely in Section \ref{sec:dyna}), that was described in
the introduction as the deposition/removal of $k$ cubes, can be seen as a
concatenation of  a random number $k\ge 1$ of elementary moves in $k$ adjacent hexagons
in the same vertical column.
We can therefore see each ``horizontal dimer/lozenge'' (we call them
``beads'' hereafter\footnote{A similar terminology was adopted in
\cite{Bout_bead} for a model where bead positions take real values:
such continuous model can be obtained from the dimer
coverings of the hexagonal lattice in the limit where the density of
horizontal dimers tends to zero, by suitably rescaling the lattice.}) as attached to a 
 ``column'' (an infinite vertical stack of hexagons): the bead can move up and down along the column but not
change column. 
The set of possible bead positions can be identified with $\mathbb Z$
on, say, even columns and with $\mathbb Z+1/2$ on odd columns.
Note that beads of neighboring columns are interlaced: if on column
$\ell$ there are two beads at positions $z_1<z_2$ then necessarily in
column $\ell-1$ there is a bead at a position $z_1<z_3<z_2$ and
similarly for column $\ell+1$.

\begin{definition}
\label{def:I+-}
For each dimer configuration $\sigma$ and bead $b$ we let
$I^+_b=I^+_b(\sigma)$ be the collection of available positions above
it, i.e. positions that $b$ can reach via a concatenation of
elementary moves that do not touch any other bead and do not violate
the interlacing constraints. We define similarly $I^-_b=I^-_b(\sigma)$
as the collection of available positions below it.
\end{definition}

\begin{remark}
\label{rem:identificazione}
Given a finite or infinite subset $\Lambda$ of $\mathcal H$, we denote $\sigma|_\Lambda$ the dimer configuration restricted to $\Lambda$ and
$\eta|_\Lambda$ the configuration of beads restricted to
$\Lambda$. If $\Lambda=\mathcal H$ we omit the index $\Lambda$.
If (as will be the case in Theorem \ref{th:losanghe} below) every column
contains at least one bead,  $\sigma$ can be reconstructed by knowing just $\eta$. In this case we will identify a dimer covering with a bead configuration.  
\end{remark}

\subsection{Height function}
\label{sec:altezza} 
On  $\mathcal H$, we take a coordinate frame where the axis $\vec e_1$ forms
a clockwise angle $+5\pi/6$ with the usual horizontal axis and the
axis $\vec e_2$ an angle $+\pi/6$, see Figure \ref{fig:losanghe}.  We
also set $\vec e_3=-\vec e_1-\vec e_2$ to be the vertical unit vector.
\begin{definition}Let $\mathcal H^*$ denote the dual graph of
  $\mathcal H$ (it is a triangular lattice, whose vertices are
  vertices of lozenges). Vertices of $\mathcal H^*$ are as usual
  identified with hexagonal faces of $\mathcal H$.

The height function $h:\mathcal H^*\mapsto \mathbb Z$ is an
integer-valued function, defined up to an arbitrary additive constant. When
moving one step in the $\vec e_1$ or $\vec e_2$ direction,  the
height increases by $1$ when a dimer (or equivalently lozenge) is crossed and stays constant
otherwise.
\end{definition}
Note that, with this convention, $h$ corresponds to \emph{minus} the
height function with respect to the horizontal plane, and observe also
that when moving one step in the $\vec e_3$ direction, $h$ decreases by
$1$ if no dimer is crossed and stays constant otherwise. 

Given  ${ \rho}=(\rho_1,\rho_2)\in\mathbb R^2$ with $0<\rho_1,\rho_2<1$, and
$0<\rho_1+\rho_2<1$ (we call ${\rho}$ a \emph{non-extremal slope})
there exists a unique translation-invariant ergodic Gibbs
state $\pi_{{\rho}}$ with slope  $\rho$. This is a translation invariant
probability law on the set of dimer coverings of $\mathcal H$, that
satisfies (cf. \cite[Sec. 6]{Klecturenotes}):
\begin{itemize}
\item $\pi_{{\rho}}$ is
ergodic with respect to translations by $a \vec e_1+b \vec e_2, a,b\in \mathbb
Z$;
\item  it satisfies the Dobrushin-Lanford-Ruelle equations: conditionally on the
dimer configuration $\sigma_{\mathcal G^c}$ outside a given finite subset  $\mathcal G\subset
\mathcal H$, $\pi_{{\rho}}$ is the uniform measure over all dimer
coverings $\sigma_{\mathcal G} $ of $\mathcal G$ compatible with
$\sigma_{\mathcal G^c}$, i.e. such that $\sigma_{\mathcal G^c}\cup
\sigma_{\mathcal G} $  is a dimer covering of $\mathcal H$;
\item it has slope $\rho$, i.e. $\pi_{{\rho}}(h_{x+\vec
    e_i}-h_{x})=\rho_i, i=1,2$.
\end{itemize}
Note that $\rho_1$ is the density of
south-east oriented lozenges, $\rho_2$ is the density of north-east
lozenges and $\rho_3:=1-\rho_1-\rho_2$ the density of horizontal lozenges.
The non-extremality requirement on ${\rho}$ means that all three types of lozenges have non-zero
density.

The measure $\pi_\rho$ is in a sense completely known and  has a
determinantal representation, that we recall here briefly (cf. in
particular \eqref{eq:multidimero}), since it
will be needed in the following. See \cite{KOS,Klecturenotes} for further details. First of
all, color sites of $\mathcal H$ white/black according to whether they
are the left/right endpoint of a horizontal edge  and let $\mathcal
H_W,\mathcal H_B$ be the sub-lattice of white/black vertices. We
denote $\rw_{0},\rb_{0}$ the black/white vertices indicated in
Figure \ref{fig:losanghe} and we let $\rw_{x},\rb_{x}$, with
$x=(x_1,x_2)\in\mathbb Z^2$, be the
translation of $\rw_{0},\rb_{0}$ by $x_1\vec e_1+x_2\vec e_2$.

Take a triangle with angles $\theta_i=\pi\rho_i, i=1,2,3$ and let $k_i,i=1,2,3$ be the length
of the side opposite to $\theta_i$. 
Define the Kasteleyn matrix $K=\{K(\rb,\rw)\}_{\rb\in\mathcal
H_B,\rw\in\mathcal H_W}$ as follows: If $\rb,\rw$ are not nearest
neighbors, then $K(\rb,\rw)=0$. If they are nearest neighbors, then
$K(\rb,\rw)=k_1$ or $k_2$ or $k_3$ according to whether the edge
$\rb\rw$ is oriented south-east,  north-east  or horizontal.

Define also the matrix $K^{-1}=\{K^{-1}(\rw,\rb)\}_{\rw\in\mathcal H_W,\rb\in\mathcal
H_B}$ as
\begin{gather}
\label{eq:K-1}
  K^{-1}(\rw_{x},\rb_{x'})=\frac1{(2\pi i)^2}\int_{\mathbb
    T}\frac{z^{-(x'_2-x_2)}w^{x_1'-x_1}}{P(z,w)}\frac{dz}{z}\frac{dw}{w}
:=\frac1{(2\pi i)^2}\int_{\mathbb
    T}\frac{z^{-(x'_2-x_2)}w^{x_1'-x_1}}{k_3+k_1 z+k_2w}\frac{dz}{z}\frac{dw}{w}
\end{gather}
where the integral is taken over the two-dimensional unit torus $\mathbb
T:=\{(z,w)\in\mathbb C^2: |z| = |w| = 1\}$. The long-distance behavior
of $K^{-1}$ is precisely known \cite{KOS}: since the polynomial $P$
has two simple zeros on the torus, $K^{-1}$ decays like the inverse of
the distance
so that  in particular
\begin{gather}
  \label{eq:64}
|   K^{-1}(\rw_{0},\rb_{x})|\le \frac {C(\rho)}{|x_1|+|x_2|+1}
\end{gather}
with $C(\rho)<\infty$
(this in general fails if $\rho$ is extremal, e.g. if only one of the three
dimer orientations has positive density).

Given a set of (not necessarily horizontal) edges
$e_1=(\rw_1,\rb_1),\dots,e_k=(\rw_k,\rb_k)$ of $\mathcal H$, the correlation function
$\pi_\rho(\delta_{e_1}\dots
\delta_{e_k})$ (with $\delta_e$ the indicator function that there is a
dimer at $e$) is given by
\begin{gather}
\label{eq:multidimero}
  \pi_\rho(\delta_{e_1}\dots\delta_{e_k})=\left(\prod_{i=1}^k
    K(\rb_i,\rw_i)\right)\det\left(K^{-1}(\rw_i,\rb_j)\right)_{1\le
    i,j\le k}.
\end{gather}
Note, also in view of formula \eqref{eq:K-1},  that the r.h.s. of \eqref{eq:multidimero} is invariant if we
multiply all $k_i$ by a common factor $c$, so that we may for instance
fix the sum $k_1+k_2+k_3$ to $1$.
\subsection{Definition of the dynamics and stationarity of Gibbs states}
\label{sec:dyna}
The dynamics is informally defined as follows (cf. Fig. \ref{fig:biglie2}). To each column $\ell$ and to each
possible bead position $z$ (horizontal edge of $\mathcal H$) we associate two independent Poisson clocks of mean
$p\in[0,1]$ and $q\in[0,1]$ respectively.
We call them $p$-clocks and $q$-clocks, with obvious meaning.  Clocks
at different locations are independent. 
When a
$p$-clock (resp. a $q$-clock) at $(\ell,z)$ rings, if $(\ell,z)$ is
occupied by a bead we do nothing. Otherwise, we look at the highest
(resp. lowest) bead
(if any) on column
$\ell$ that is at position lower (resp. higher) than $z$: if it can be moved
to $z$ without violating the interlacing constraints then we do so, otherwise we do
nothing.

\begin{figure}[h]
  \includegraphics[width=4cm]{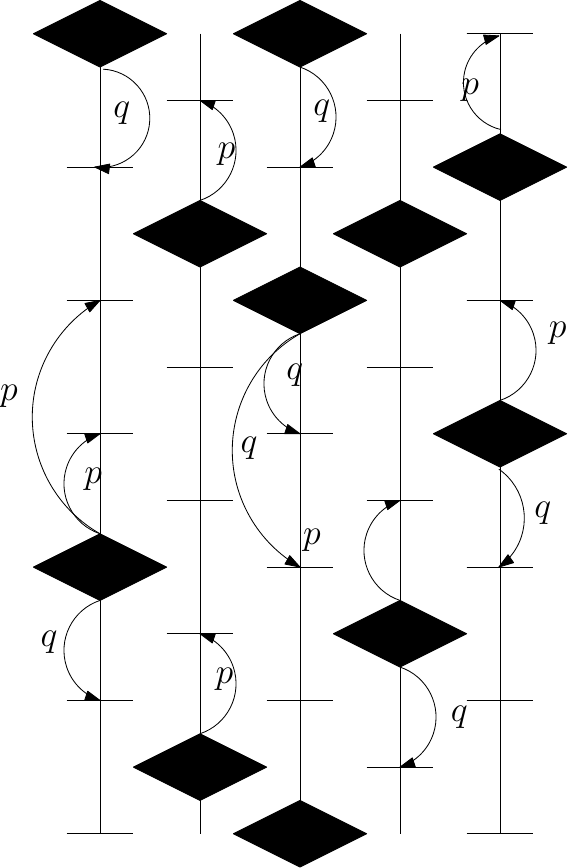}
\caption{A portion of the lattice with the allowed moves and the
  respective rates. Beads are drawn as black lozenges. When $p=1,q=0$
  or $p=0,q=1$ the process can
be seen as an infinite set of discrete Hammersley
processes, one per column, each interacting with the two
neighboring processes. If instead we allowed particles to jump only
by $\pm1$ with rates $p/q$, the process would be the single-flip
dynamics described in the introduction and would correspond to an
infinite collection of mutually interacting asymmetric simple
exclusion processes (ASEP), one per column.}
\label{fig:biglie2}
\end{figure}

It is not obvious that the process is well defined on the infinite
lattice. The danger is that
beads could escape to $+\infty $ or to $-\infty$ in finite
time (even in an arbitrarily small time). This may occur when 
spacings between beads
in the initial configuration
grow sufficiently fast at infinity.
The problem is that the
rate at which a bead moves, say, upward is $p\times |I^+_b|$ and the average
size of  the jump is $(|I^+_b|+1)/2$,  and $|I^+_b|$ is not
bounded. 

Our first result (Theorem \ref{th:losanghe}) is that the process is well defined for almost every
initial condition sampled from  $\pi_{{\rho}}$ and that $\pi_{{\rho}}$ is
invariant. ``Well-defined'' means that the
displacement of every bead with respect to its position at time zero
is almost surely finite for every $t\ge0$.
In the symmetric case $p=q$, assuming that the process is
well defined, invariance of the Gibbs measure is obvious because
it is reversible.

\smallskip

To precisely formulate the result, let us start by defining, given
$K=(K_p,K_q)\in (\mathbb R^+)^2$, a cut-off process where $p$-clocks
at distance more than $K_p$ (resp. $q$-clocks at distance more than
$K_q$) from the origin of $\mathcal H$ are switched off.   As long as
$K_p,K_q<\infty$ there is no problem in defining the process on the
whole $\mathcal H$, since this is effectively a Markov jump process on
a finite state space (once a particle is inside the ball of radius
$\max(K_q,K_p)$ it cannot leave it and therefore there is only a
finite number of particles, determined by the initial condition, that
can ever move).  We call $X^\sigma_{t;K}$ the configuration at time
$t$, started from initial condition $\sigma$.  Given a column $\ell$,
let $z_t(\ell,n;K)$ be the position of its $n^{th}$ bead at time $t$,
with $z_t(\ell,n;K)<z_t(\ell,n+1;K)$. The label $n$ is assigned in the
initial condition and is attached to beads forever.  For instance, one
can assign the label $(\ell,0)$ to the lowest bead in $\ell$ with
non-negative vertical coordinate (in the initial condition). We assume
hereafter that in each column there is a doubly infinite set of beads,
i.e. the index $n$ runs over all of $\mathbb Z$.

Two processes
with different cut-offs $K$ and $K'$ can be coupled in the obvious way:
their $p$-clocks (resp. $q$-clocks) are the same in the ball of radius
$\min(K_p,K'_p)$ (resp. $\min(K_q,K'_q)$).
It is then easy to check that $z_t(\ell,n;K)$ is increasing w.r.t. $K_p$ and
decreasing w.r.t. $K_q$. 
We will then define
\begin{eqnarray}
  \label{eq:37}
 z_t(\ell,n)=\lim_{K_q\to\infty}\lim_{K_p\to\infty} z_t(\ell,n;K)
\end{eqnarray}
to be the position of the $(\ell,n)$-th bead at time $t$ for the process
without cut-off.

Assuming that $z_t(\ell,n)$ is finite for every $(\ell,n)$, 
call   $X_t$ the corresponding bead 
   configuration 
and let
$\mathbb P_\nu$ be the law of the process $(X_t)_{t\ge0}$
  started with initial distribution $\nu$ (if $\nu$ is concentrated at
  some $\sigma$, then we write just $\mathbb P_{\sigma}$). 
\begin{theorem}
\label{th:losanghe}
For almost every initial condition sampled from $\pi_{{\rho}}$,
with ${\rho}$ a non-extremal
  slope,  the limit \eqref{eq:37} is almost surely finite for all $(\ell,n)$ and
  $t\ge 0$. Moreover, $\pi_{{\rho}}$
is invariant. More
  precisely, if $f$ is a local bounded function of the dimer configuration 
  one has for every $t\ge0$
  \begin{eqnarray}
    \label{eq:32}
 \mathbb E_{\pi_{\rho}}(f(X_t))= \int  \pi_{{\rho}}(d\sigma) \mathbb
 E_\sigma f(X_t)= \pi_{{\rho}}(f).
  \end{eqnarray}
\end{theorem}
Here, a function $f$ is said to be local if it depends only on
$\sigma_\Lambda$ for some finite $\Lambda$.
It is also possible to see (cf. Remark \ref{rem:scambio}) that
 the limit \eqref{eq:37} does not depend on the order how one takes
 the limits
  $\lim_{K_q\to\infty}$ and $\lim_{K_p\to\infty}$.

Theorem \ref{th:losanghe} is proven partly in Section \ref{sec:well-d} (existence of the
dynamics)  and partly in Section \ref{sec:parte2}
(invariance of $\pi_\rho$).

\begin{remark}
  With this result in hand, it is clear that one can construct many
  other driven processes that leave $\pi_\rho$ invariant, simply
  adding to the generator of the bead dynamics another generator
  $\mathcal L$ with  respect to which $\pi_\rho$ is reversible (for instance, $
  \mathcal L$ could be the
  generator of the single-flip dynamics with symmetric rates).
\end{remark}

\medskip
It is a relatively standard fact to deduce from Theorem \ref{th:losanghe}
that, if we start from
$\pi_{\rho}$ conditioned to have a bead say at the origin, then the
law of the dimer
configuration re-centered at the time-evolving position of this
marked bead (\emph{tagged particle}) 
is time-independent, see Section \ref{sec:palm}.
More precisely, 
fix a horizontal edge $e_0$ of $\mathcal H$. Given an initial condition $\sigma$
such that there is a bead at $e_0$, call $\phi_t$ its vertical coordinate at time
$t$ (the horizontal coordinate does not change). Let also $\hat
X_t:=\tau_{\phi_t-\phi_0} X_t$, with $\tau_x$ the vertical translation by 
$x\in \mathbb Z$, be the dimer configuration viewed from the tagged
bead and call $\hat {\mathbb P}_\nu$ the law of the process $(\hat X_t)_{t\ge0}$
started from some initial distribution $\nu$.
Finally, let 
$\hat \pi_{\rho}$ be the Gibbs measure $\pi_{\rho}$ conditioned on the
event that there is a bead at $e_0$. 

\begin{proposition}
\label{th:palm}
  The measure $\hat \pi_{\rho}$ is invariant for the dynamics of the dimer configuration viewed from the tagged
bead: for every  bounded local function $f$ and $t\ge0$, 
\begin{eqnarray}
  \label{eq:51}
  \hat\pi_{\rho}(f(\hat X_t))=\hat \pi_{\rho}(f).
\end{eqnarray}
\end{proposition}

\section{Interface speed and fluctuations}

\label{sec:intfluctu}

The stationary states $\pi_{\rho}$ are characterized by an upward or downward flux of
beads, according to whether $p>q$ or $p<q$.
The particle flux is directly related to the average height increase
in the stationary state. 
While the height function was defined only up to an additive constant, one
can define unambiguously the increase of the height at a face $x$ from time
$0$ to $t$: $Q_x(t):=h_x(t)-h_x(0)$ equals the number of beads that 
cross the face $x$ downward up to time $t$, minus the number of beads that cross it upward.

For each horizontal bond $e$ let
$b^+(e)$ (resp. $b^-(e)$) be the lowest (resp. highest) bead in the column of $e$, at vertical
position strictly higher (resp. strictly lower) than $e$. 
Also, call $V(e,\uparrow)$ the collection of hexagons that
$b^-(e)$ has to cross to reach position $e$  and set
$V(e,\uparrow)=\emptyset$ if this move is not possible (keeping the
other beads where they are).
Define $V(e,\downarrow)$ similarly.

The following result identifies the average height drift and shows
that the fluctuations of $Q_x(t)$
in the stationary measure
are smaller than any power of $t$:
\begin{theorem}
\label{th:vvar}
  For any face $x$,
\begin{gather}
\label{eq:Qx}
  \mathbb E_{\pi_\rho}(Q_x(t))=t\,(q-p)\,J
\end{gather}
with 
\begin{gather}
\label{eq:J}
  J=\pi_\rho(|\{e:x\in V(e,\uparrow)\}|)>0.
\end{gather}
For every $\delta>0$, 
\begin{gather}
\label{eq:varianzaQ}
\lim_{t\to\infty}  \mathbb P_{\pi_\rho}(|Q_x(t)-\mathbb E_{\pi_\rho}(Q_x(t))|\ge
  t^\delta)=0.
\end{gather}
\end{theorem}
 Note that only edges $e$ in the same column as $x$ and above it can
contribute to $J$. The value of $J$ is independent of $x$ by
translation invariance of $\pi_\rho$. The r.h.s. of \eqref{eq:Qx} is linear in $t$ because of stationarity
of $\pi_\rho$ and linear in $(q-p)$ because the stationary state
$\pi_\rho$ does not depend on $p,q$.

Theorem \ref{th:vvar} is proven in Section \ref{sec:velfluct}.

\medskip

It is not obvious to compute $J$ explicitly in terms of
the slope $\rho$, starting directly
from the determinantal representation of the Gibbs states. 
In \cite{BF1,BF2}, A. Borodin and P. L. Ferrari
considered the dynamics for $p=0,q=1$ for a special, ``integrable'', initial condition $\omega$, whose 
height function $(h_0(x))_{x\in\mathcal H^*}$ is deterministic and has
non-constant slope (see Fig. 1.2 of \cite{BF1}: lozenges with a dot
correspond to our south-east oriented lozenges, white squares to
our north-east lozenges, while dark lozenges   correspond to our beads). 
Let us emphasize that with such initial condition, each bead has a
\emph{deterministic} lowest position it can possibly reach on its
column (this is related to the fact that in \cite[Fig. 1.1, 1.2]{BF1}
there is no dotted lozenge with coordinate $n<0$), so that the
well-posedness of the process poses no problem in that case.
One of the results of \cite{BF1} is a hydrodynamic limit, that in our notations we
can formulate as follows:  for every $\xi,\eta$
and $\tau>0$
one has 
\begin{gather}
\lim_{L\to\infty}\frac1L  \mathbb E_{\omega} \left[h_{(\lfloor \xi L\rfloor ,\lfloor \eta L\rfloor)}(\tau L)\right]={\bf h}(\xi,\eta,\tau)
\end{gather}
and ${\bf h}$ satisfies 
\begin{gather}
\label{eq:V}
  \partial_\tau{\bf h}=V(\partial_\xi {\bf h},\partial_\eta{\bf h})
  =\frac1\pi\frac{\sin(\pi\partial_\eta{\bf h})\sin(\pi
     \partial_\xi{\bf
       h})}{\sin(\pi(\partial_\eta{\bf
       h}+\partial_\xi{\bf h}))}
\end{gather}
(this corresponds to formulas (1.9)-(1.11) in \cite{BF1}, after 
after a suitable change of coordinates due to the fact that in \cite{BF1,BF2} the height is not
taken with respect to the horizontal plane and a different reference
frame than our $\vec e_1,\vec e_2$ frame is used). 
From this, one can naturally guess that $J$ in \eqref{eq:J} should be given
by 
 \begin{gather}
 \label{eq:Jesplicita}
   J=\frac1\pi\frac{\sin(\pi\rho_1)\sin(\pi\rho_2)}{\sin(\pi(\rho_1+\rho_2))}.
 \end{gather}
 Since $\rho_1, \rho_2 $ and $\rho_1+\rho_2$ are in $(0,1)$, the above
 expression is immediately seen to be positive.  After a first version
 of this work was completed, Chhita and Ferrari \cite{ChhitaFerrari}
 proved, through a smart combinatorial identity based on the
 determinantal structure of the Gibbs states, that indeed
 \eqref{eq:Jesplicita} holds.

By the way, Proposition 3.2 of \cite{BF1} says that the law of local dimer
observables around point 
$(\lfloor \xi L\rfloor,\lfloor \eta
L\rfloor)$ at  time $\tau L$ tends as $L\to\infty$ to that of the same
observables under the Gibbs state of slope
$\rho=(\partial_\xi {\bf h}(\xi,\eta,\tau),\partial_\eta {\bf
  h}(\xi,\eta,\tau))$. On the basis of this, it was natural to
conjecture that our Theorem \ref{th:losanghe} holds.

\medskip

  Referring to \eqref{eq:varianzaQ}, we believe that the order of magnitude of the variance of $Q_x(t)$
  is actually $\log t$: this is indeed the result found by Borodin and
  Ferrari \cite{BF1,BF2}, in the particular case where $p=0$, $q=1$
  and for the special initial
  condition $\omega$ mentioned above. 
In this respect, our  method allows indeed to refine  estimate
\eqref{eq:varianzaQ}, under a  (purely technical, we believe)
condition on the slope $\rho$, to the following:
\begin{theorem}
\label{th:32}
If the
slope $\rho$ satisfies 
\begin{gather}
  \label{eq:ciro}
\sqrt{k_1 k_2} C(\rho)<1,
\end{gather}
with $C(\rho)$ defined in \eqref{eq:speranum} and $k_1,k_2$ as in \eqref{eq:K-1},
  we have for some $c<\infty$
\begin{gather}
  \label{eq:varianzaQ1}
\limsup_{t\to\infty}\mathbb P_{\pi_\rho}(|Q_x(t)-\mathbb E_{\pi_\rho}(Q_x(t))|\ge u\sqrt
{\log t})\le \frac c{u^2}.
\end{gather}
\end{theorem}

For instance, if $\rho=(1/3,1/3)$ (the three
types of dimers have density $1/3$, in which case  $k_1,k_2,k_3$ are all equal) one finds, evaluating
numerically the integral in \eqref{eq:speranum}, that the l.h.s. of
\eqref{eq:ciro} is $0.896...<1$, so that \eqref{eq:varianzaQ1} holds. By continuity, this
remains true in a whole neighborhood of $\rho=
(1/3,1/3)$ while, again numerically, \eqref{eq:ciro}
does not seem to be satisfied in the whole set of non-extremal slopes $\rho$. 

Let us stress once more that we believe \eqref{eq:varianzaQ1} to
hold for every non-extremal $\rho$ and to be of the optimal order
w.r.t. $t$, while we do not attach any particular
meaning to condition \eqref{eq:ciro}.

\begin{remark}
  It is possible to define alternatively the stationary drift
  as follows. 
Sample $\sigma$ from $\hat\pi_{{\rho}}$ and call as
above $\phi_t$ the vertical coordinate of the tagged bead $b_0$ at time
$t$. 
From Proposition \ref{th:palm} it is easy to deduce that the average
of $\phi_t-\phi_0$ is exactly linear in $t$, while from the definition
of the process and the fact that $|I^-_{b_0}|$ has
the same law as $|I^+_{b_0}|$ (the Gibbs measures are invariant by
reflection through the center of any given hexagonal face; this
follows e.g. from uniqueness of $\pi_\rho$ given the slope) one sees
\begin{gather}
  \label{eq:53}
  v:=\frac1t\int \hat\pi_{\rho}(d\sigma) \mathbb E_\sigma
  (\phi_t-\phi_0)=(p-q)\hat\pi_{\rho}(|I^+_{b_0}|(|I^+_{b_0}|+1)/2).
\end{gather}
It is not hard to deduce from the stationarity of $\pi_{\rho}$ that
\begin{gather}
\label{eq:alternativa}
\mathbb E_{\pi_{\rho}}(h_x(t)-h_x(0))=-t\,\rho_3\, v,  
\end{gather}
where we recall that $\rho_3$ is the density of beads (the reason for
the minus sign is that when a bead moves upward the height function
decreases).  Indeed, suppose for simplicity that $p=1,q=0$. The
l.h.s. of \eqref{eq:alternativa} equals minus the
 sum over the the edges $e$ below $x$ of the probability that there is
 a bead at $e$ at time zero and that at time $t$ is has moved at least
 $n+1$ steps up, with $n\ge0$ the number of hexagonal faces between $e$
 and $x$. By translation invariance of $\pi_\rho$, this equals
 \begin{gather}
  -\rho_3 \sum_{n\ge0}\int \hat\pi_\rho(d\sigma)\mathbb
  P_\sigma(\phi_t-\phi_0>n)= - \rho_3 \int
  \hat\pi_\rho(d\sigma)\mathbb E_\sigma(\phi_t-\phi_0)=-t \, \rho_3\,v
 \end{gather}
where we used positivity of $\phi_t-\phi_0$ in the first equality and
\eqref{eq:53} in the second.
\end{remark}

\subsection{Extension to dominos (perfect matchings of $\mathbb Z^2$)}

\label{sec:estensionedomini}

Our result extends to perfect matchings of $\mathbb Z^2$, or equivalently domino
tilings of the plane, cf. Fig. \ref{fig:domini}: also in this case, one can define an asymmetric
Markov dynamics (the height function has a non-zero drift) that leaves
the Gibbs states invariant. We give only a sketchy description of the
generalization, omitting those details that are identical to the case
of the honeycomb lattice. 

\begin{figure}[h]
  \includegraphics[width=8cm]{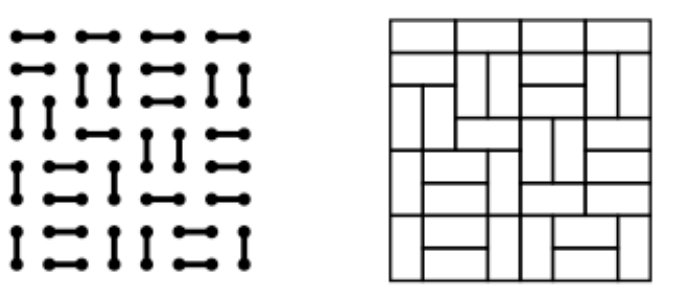}
\caption{The correspondence between dimer covering (perfect matching)
  of $\mathbb Z^2$ and domino tiling of the plane. }
\label{fig:domini}
\end{figure}

Since $\mathbb Z^2$ is bipartite we can color its vertices black/white
with the rule that each vertex has neighbors only of the opposite
color.
The height function $h$ on the set of faces of $\mathbb Z^2$ can be
defined (modulo an arbitrary additive constant) as follows: for each
$x,y$ choose any nearest-neighbor path $C_{x\to y}$
from $x$ to $y$ and set
\begin{gather}
  h_y-h_x=\sum_{e\in C_{x\to y}}\sigma_e (\delta_e-1/4)
\end{gather}
with the sum running over the edges crossed by the path, $\sigma_e=\pm1$ according to whether $C_{x\to y}$ crosses $e$
with the white vertex on the left/right and $\delta_e$ the indicator
function that there is a dimer at $e$. The definition is independent of
the choice of path.

The classification of translation-invariant ergodic Gibbs states is
analogue to the honeycomb lattice case (actually the structure is the
same for all planar, periodic, infinite bipartite graphs \cite{KOS}):
there exists an open polygon $P \subset\mathbb R^2$ (for the lattice
$\mathbb Z^2$ it is a square, while for $\mathcal H$ it is a triangle,
as discussed in Section \ref{sec:altezza}) 
such that for every $\rho=(\rho_1,\rho_2)\in P$
(non-extremal slope) there exists a
 unique translation-invariant ergodic Gibbs state $\pi_\rho$
satisfying
\begin{gather}
  \pi_{\rho}(h_{x+\vec e_i}-h_{x})=\rho_i, \; i=1,2,
\end{gather}
where the vectors $\vec e_i$ are as in Figure \ref{fig:threads}.
The determinantal representation \eqref{eq:multidimero} still holds,
with a different polynomial $P(z,w)$ that however still has two simple
zeros on the torus $\mathbb T$.
\medskip
In order to define the irreversible dynamics that leaves the Gibbs
states invariant, we have to find an analogue of the ``columns'' and 
``beads''. This is inspired by \cite{cf:LRS,cf:LT}.
The set of square faces of $\mathbb Z^2$ is sub-divided into infinite
``columns'' (indexed by $\ell\in \mathbb Z$), i.e. diagonally oriented
zig-zag paths, see Figure \ref{fig:threads}. 
Dimers that occupy an edge across a column are called ``beads''. 
Each column is oriented along the positive $\vec e_1$ direction, so it
makes sense to say that a bead $b_1$ in column $\ell$ is above a bead
$b_2$ in the same column.

\begin{figure}[h]
  \includegraphics[width=8cm]{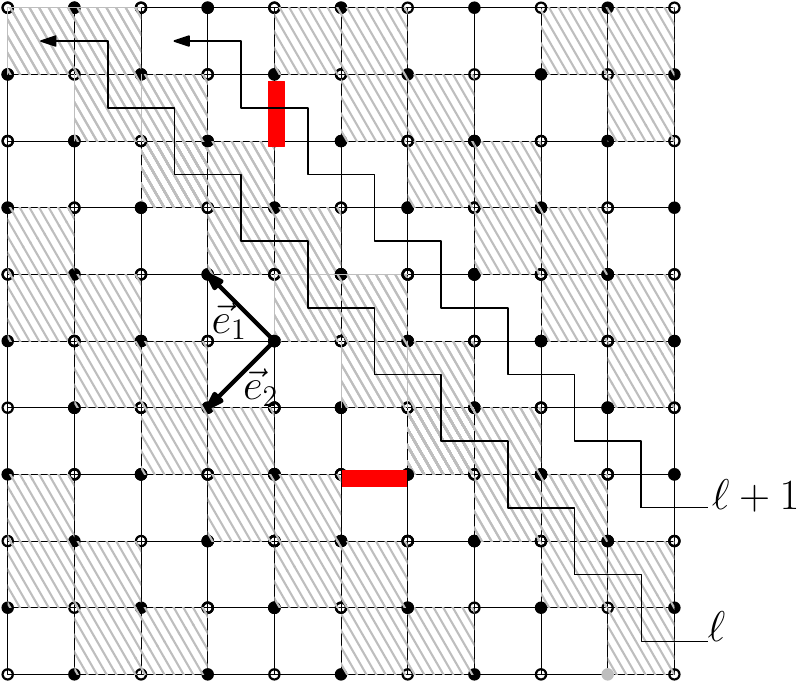}
\caption{A portion of $\mathbb Z^2$ with vertices colored
  black/white. The colored graph is periodic, i.e. invariant
  w.r.t. translations by $n_1\vec e_1+n_2\vec e_2, n_1,n_2\in\mathbb
  N$. The white and shaded zig-zag regions are the ``columns'', indexed by
  $\ell$. Columns are oriented in the positive $\vec e_1$
  direction. The two thick bonds, transversal to columns, represent
  two beads.  }
\label{fig:threads}
\end{figure}

Given columns $\ell,\ell+1$, call $Y_{\ell}$ the
set of vertices of $\mathbb Z^2$ shared by the two columns and 
order the sites of $Y_{\ell}$ according their $\vec e_1$ coordinate. Then, a bead
$b$ on column $\ell$ is said to be higher than a bead $b'$ on $\ell+1$
if the vertex of $b$ on $Y_\ell$ is higher than the vertex of $b'$ on $Y_\ell$.
With this definition, it is easy to see that 
beads satisfy the same interlacement property as on the honeycomb
graph: given beads $b_1,b_2$ on $\ell$, there exists $b_3$ on $\ell-1$
and $b_4$ on $\ell+1$ with $b_1<b_3<b_2$ and $b_1<b_4<b_2$.
Also, like on the  honeycomb lattice, it is easy to see that if there 
is at least a bead in each column, then it is possible to reconstruct
the whole dimer covering knowing only the bead positions.

\medskip

The dynamics is then defined as follows. Assign to any possible bead
position, i.e. to each edge that is transversal to some column, two
independent Poisson clocks of rates $p$ and $q$, as before. All clocks
are independent. When a $p$-clock (resp. $q$-clock) at edge $e$ of column
$\ell$ rings, if there is a bead at $e$ then do nothing. Otherwise,
move the first bead below (resp. above) $e$ in column $\ell$ to
position $e$, provided this does not violate the interlacing
constraints. Note that the dynamics is the same as on the honeycomb
lattice, only the definition of ``column'' and ``bead'' being
lattice-dependent. Observe also that each move can be seen as a
concatenation of elementary moves on $n$ adjacent faces along the same
column, each elementary move consisting in the rotation by $\pi/2$ of
two dimers on the same face of $\mathbb Z^2$
(Fig. \ref{fig:flipdomino}).
\begin{figure}[h]
  \includegraphics[width=4cm]{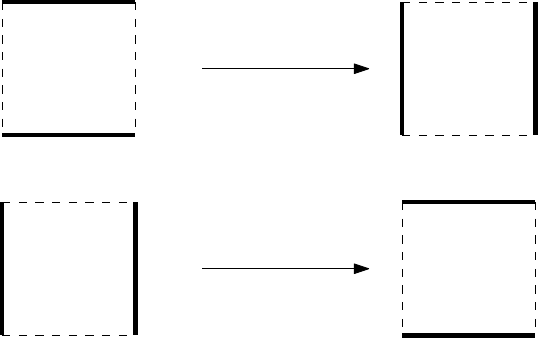}
\caption{The two elementary moves of the domino dynamics }
\label{fig:flipdomino}
\end{figure}
In fact, the effect of an elementary move is
to shift a single bead one position up or down along its column. Note that,
like in the case of the hexagonal lattice,
when a bead moves one step upward crossing a face $f$, the height
function at $f$ changes by $-1$.

As in Section \ref{sec:intfluctu}, given an edge $e$ transversal to some column
$\ell$, call $b^-(e)$ the highest bead in $\ell$, strictly lower than
$e$ and let $V(e,\uparrow)$ the collection of square faces of $\ell$
that $b^-(e)$ crosses when it is moved to $e$ (with
$V(e,\uparrow)=\emptyset$ if the move is not allowed). Then:
\begin{theorem}
The claim of Theorems \ref{th:losanghe} and \ref{th:vvar} hold also
for the bead dynamics on dimer coverings of $\mathbb Z^2$. \end{theorem}

With the exception of Section \ref{sec:toroquadrato}, in the rest of
the work we will always consider the case of the hexagonal lattice.

\section{Dynamics on the torus}
\label{sec:sultoro}
\subsection{Honeycomb lattice}
We will let the torus $\mathbb T_L$ denote the hexagonal graph $\mathcal H$, periodized (with
period $L$) along directions $\vec e_2,\vec e_3$ and we assume that
$L\ge 3$.  
Note that now columns $\ell $ along which beads move are $L$
``circles'' containing  $L$ hexagonal faces. We will say as before that a
bead moves ``upward'' or ``downward'', but what we mean is that it
moves in the positive or negative $\vec e_3$ direction around the torus.

Let $N^L_{{\rho}}$ be the
set of configurations such that  the height changes by $\lfloor \rho_2
L\rfloor$ (resp. $\lfloor L\rho_3\rfloor -L$)
 along
any closed path 
winding once in the positive $\vec e_2$ (resp. $\vec e_3$) direction.
 On each column $\ell$ there are
$\lfloor \rho_3 L\rfloor$ beads and bead positions on
neighboring columns are again interlaced.
 We denote $\pi^L_{{\rho}}$ the
 uniform measure over $N^L_{{\rho}}$.
It is known that $\pi^L_{{\rho}}$ converges weakly to
$\pi_{{\rho}}$, if the configuration space is equipped with the
product topology \cite{KOS}. Essentially, averages of bounded
local functions converge.

\medskip

On $\mathbb T_L$ the process is defined similarly as in Section
\ref{sec:dyna} for the infinite graph.  For instance, when a $p$-clock at an
edge $e$ rings, one moves to $e$ the first bead that is
found when proceeding in the $-\vec e_3$ direction from $e$ along the same column, unless this move is
forbidden by 
the interlacing constraint.
The process is ergodic on $N^L_{{\rho}}$, actually it is known that  we
  can go from any configuration to any other by positive-rate elementary moves as in
  Fig. \ref{fig:flip} (see \cite[Lemma 1]{CT} for details).
\begin{proposition}
\label{th:invarianzatoro}
  The measures  $\pi^L_{{\rho}}$  
  are 
  stationary.
\end{proposition}
It is actually easy to deduce, using ergodicity of the process in each
of the sectors $N^L_\rho$, that the only stationary measures are
convex combinations of $\pi^L_\rho$.

\begin{proof}[Proof of Proposition \ref{th:invarianzatoro}]
  
Call $\mathcal L^L$ the
generator of the process. We want to check that 
\[
\pi^L_{{\rho}}\mathcal L^L=0
\]
(stationarity of $\pi^L_{{\rho}}$)
. One can decompose the generator as $\mathcal
L^{+,L}+\mathcal L^{-,L}$ with $\mathcal
L^{+,L}$ involving only the up-jumps (related to the $p$-clocks) and $\mathcal L^{-,L}$ the
down-jumps. It is sufficient to prove that
$\pi^{L}_{{\rho}}\mathcal L^{+,L}=0$, for $\mathcal L^{-,L}$ the argument
being the same.
For every $\sigma\in N^L_{{\rho}}$ we have
$\pi^L_{{\rho}}(\sigma)=1/|N^L_{{\rho}}|$. Given $\sigma \in N^L_{{\rho}}$ let
$\Omega_\sigma$ be the collection of $\sigma'\in
N^L_{{\rho}}$  that can be reached from $\sigma$ by a single non-zero
up-jump (not necessarily of length one) of a
bead and let $\Omega^{(-1)}_\sigma$ be the collection of $\sigma'\in
N^L_{{\rho}}$ from which one can reach $\sigma$ with a single non-zero up-jump of a
bead. For every $\sigma'\in \Omega^{(-1)}_\sigma$ we have $\mathcal
L^{+,L}(\sigma',\sigma)=p$, while $\mathcal
L^{+,L}(\sigma,\sigma)=-p|\Omega_\sigma|$ simply because the sum of row
elements of the generator is zero.
We see then
\begin{eqnarray}
  \label{eq:18}
 \left[ \pi^L_{{\rho}}\mathcal
  L^{+,L}\right](\sigma)=\sum_{\sigma'}\pi^L_{{\rho}}(\sigma')\mathcal
  L^{+,L}(\sigma',\sigma)=
\frac p{|N^L_{{\rho}}|}(|\Omega^{(-1)}_\sigma|-|\Omega_\sigma|).
\end{eqnarray}
We want to see that $|\Omega^{(-1)}_\sigma|=|\Omega_\sigma|$. Note
that $|\Omega_\sigma|=\sum_b |I^+_b|$ while
$|\Omega^{(-1)}_\sigma|=\sum_b |I^-_b|$,
with the sum running over beads and $I^\pm_b$ being as in Definition
\ref{def:I+-}\footnote{At the expense of being pedantic let us
  emphasize that, on the torus, the set of positions available
  ``above'' a bead means the set of positions reachable via moves in
  the $+\vec e_3$ direction.}.
We will prove that $K_\sigma:=|\Omega^{(-1)}_\sigma|-|\Omega_\sigma|$ is
independent of $\sigma$: as a consequence, it must be zero because the
sum over $\sigma$ of \eqref{eq:18} is zero.
Assume that $\sigma'$ differs from $\sigma$ only by a single elementary
up-move of some bead $b$ on some column $\ell$. Then, after the move
the only beads $b'$ that may have changed their values of $I^\pm_{b'}$ are $b$
itself and $b^\pm_{\ell\pm1}$, with $b^+_{\ell+1}$ the bead in column
$\ell+1$ that is ``just above $b$'' 
\begin{figure}[h]
  \includegraphics[width=12cm]{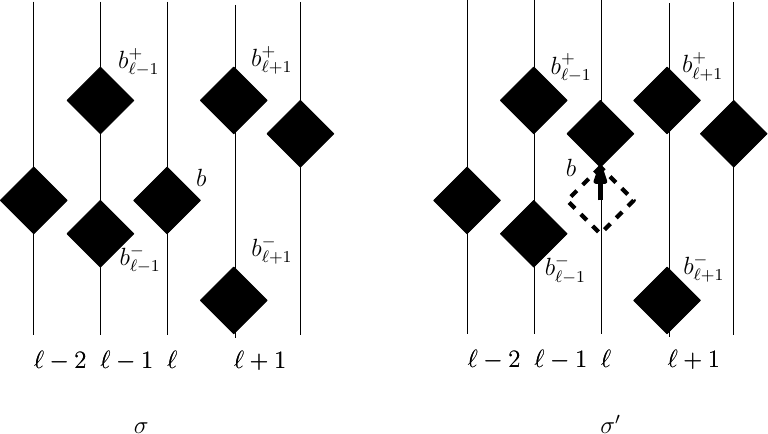}
\caption{The configurations $\sigma$ and $\sigma'$ around $b$. In this
example, once the bead $b$ moves up by $1$, $I^-_{b^+_{\ell-1}}$
decreases by $1$, $I^+_{b^-_{\ell+1}}$ increases by $1$ while
$I^-_{b^+_{\ell+1}}$ and $I^+_{b^-_{\ell-1}}$ stay constant.}
\label{fig:beads1}
\end{figure}
(see Figure \ref{fig:beads1}) and
analogously for the others.
It is clear that the contribution of $b$ to $K_{\sigma'}-K_\sigma$ is
$+2$: indeed, $|I^+_b|$ decreases by $1$ and $|I^-_b|$ increases by
$1$. Then look at column $\ell-1$. One of the following two mutually
exclusive cases occurs (Fig. \ref{fig:beads1}): either
$|I^+_{b^-_{\ell-1}}|$ increases by $1$ and $|I^-_{b^+_{\ell-1}}|$ stays
constant or $|I^+_{b^-_{\ell-1}}|$ stays constant and
$|I^-_{b^+_{\ell-1}}|$ decreases by $1$. In both
cases, the net variation of $K_{\sigma'}-K_\sigma$ from column
$\ell-1$ is $-1$. The same holds for column $\ell+1$ (since we are
assuming $L\ge 3$, columns $\ell\pm1$ are distinct). Altogether,
$K_{\sigma'}-K_\sigma=0$.
We have proved that $K_\sigma$ is unchanged if we perform an
elementary up-move. Given that the space state is connected, we proved
that $K_\sigma$ is constant (and therefore zero) on $N^L_{{\rho}}$.
\end{proof}

The analog of Proposition \ref{th:palm} for the dynamics on the torus
is the following.
\begin{proposition}
\label{th:palmtoro}
Fix a horizontal edge $e_0$ on $\mathbb T_L$, let $\sigma$ be a configuration such
that there is a bead at $e_0$ and call $\phi_t$ be the vertical
position of this bead at time $t$. 
  Let $\hat\pi^L_{\rho}$  
  be $\pi^L_{\rho
    }$ 
    conditioned to the event that there is
  a bead at $e_0$. The law $\hat\pi^L_{\rho}$ 
is
    stationary for the re-centered process $\tau_{\phi_t-\phi_0} X_t$.
\end{proposition}
\begin{proof}
  The
  proof is very similar to that of Proposition
  \ref{th:invarianzatoro}. Call $\hat {\mathcal L}^{+,L}$ the part of the
generator of the process involving only $p$-clocks. We have to show for every $\sigma$
  \begin{eqnarray}
    \label{eq:52}
   \hat \pi^L_{\rho}\hat {\mathcal L}^{+,L}(\sigma)=0.
  \end{eqnarray}
A symmetric argument then gives $ \hat \pi^L_{\rho}\hat {\mathcal
  L}^{-,L}(\sigma)=0$.

The measure $\hat \pi^L_{\rho}$ is uniform among the $|\hat N^L_{\rho
  }|$ configurations with a bead at $e_0$. 
We have $\hat {\mathcal L}^{+,L}(\sigma,\sigma)$ equal $-p$ times the
number of configurations $\sigma'$ different from $\sigma$ that can be
reached from $\sigma$
with a single move. The configuration can change either because a bead
different from $b_0$ (the bead that is at $e_0$) moves, or because
$b_0$ itself moves and then the dimer configuration has to be
re-centered around the new tagged particle position.  Note indeed that, when $b_0$ moves, necessarily the configuration
viewed from it changes, since the distance from the first bead above
it decreases. The number of reachable configurations is then
$\sum_{b\ne b_0} |I^+_b|+|I^+_{b_0}|=\sum_{b} |I^+_b|$. 
Similarly, one sees that 
\[
\sum_{\sigma'\ne \sigma}\hat \pi^L_{\rho}(\sigma')\hat {\mathcal
  L}^{+,L}(\sigma',\sigma)=\frac p{|\hat N^L_{\rho}|}\left(\sum_{b\ne
    b_0}|I^-_b|+|I^-_{b_0}|\right)=\frac p{|\hat N^L_{\rho}|}\sum_{b}|I^-_b|.
\]
Then, the l.h.s. of \eqref{eq:52} equals the r.h.s. of  
\eqref{eq:18} (only with $1/|N^L_{\rho}|$ replaced by $1/|\hat N^L_{\rho}|$), that we know to be zero. 
\end{proof}

\subsubsection{A ``gradient condition''}
\label{sec:gradiente}
The bead dynamics on the torus has an trivial conserved quantity: the number of
particles. There is however a less obvious  one. 
For each of the $L$ columns $\ell=1,\dots,L$ define
\begin{eqnarray}
  X^{(\ell)}=\sum_{n}(|I^+_{n,\ell}|-|I^-_{n,\ell}|),
\end{eqnarray}
with the sum running over the beads of column $\ell$.
We have seen in the proof of Proposition \ref{th:invarianzatoro} that the ``total charge'' $X=\sum_\ell X^{(\ell)}$ is
exactly zero. 
A simple computation shows that, when $p=q$, the instantaneous
drift of $X^{(\ell)}$ is
\begin{eqnarray}
  \lim_{\delta\to0}\frac1\delta\mathbb
  E(X^{(\ell)}(\sigma_{t+\delta})-X^{(\ell)}(\sigma_{t})|\sigma_s,s\le
  t)=(Z^{(\ell)}-Z^{(\ell-1)})(\sigma_t)-(
Z^{(\ell+1)}-Z^{(\ell)})(\sigma_t)
\end{eqnarray}
with 
\begin{eqnarray}
  Z^{(\ell)}=-\frac p2\sum_{n}(|I^+_{n,\ell}|(|I^+_{n,\ell}|+1))+\frac p2\sum_{n}(|I^-_{n,\ell}|(|I^-_{n,\ell}|+1)).
\end{eqnarray}
This is a ``gradient condition'' \cite{Spohn}: the derivative of the
charge at $\ell$
is given by the divergence of a current, here $Z^{(\ell)}-Z^{(\ell-1)}$,
which is itself the gradient of a function $Y$ of the configuration.

As we mentioned in the introduction, conditions of this type are
typically the key to guarantee that a reversible Gibbs measure
remains invariant once an external driving field that breaks reversibility is introduced, see
e.g. \cite[Sec. 2.5]{Bertini} and \cite{KLS}. 
The unusual fact here (with respect to the more standard framework of
e.g. the simple exclusion or zero range processes) is that the current
associated to the local charge
$X_{n,\ell}:=(|I^+_{n,\ell}|-|I^-_{n,\ell}|)$ does not seem to satisfy
a gradient condition, while that of the non-local charge $X^{(\ell)}$ (integrated along the
columns)
does.
Note that on the infinite
lattice $X^{(\ell)}$ is not well-defined (it is just
infinite). 

\subsection{Square lattice}
\label{sec:toroquadrato}
The finite graph $\mathbb T_L$ with periodic boundary conditions is
defined like for the honeycomb lattice, except that the directions
along which one periodizes are now $\vec e_1,\vec e_2$, see 
Fig. \ref{fig:threads}. Note that each periodized column is a
``circle'' containing $2L$ square faces.
The measure $\pi^L_\rho$ is defined as the uniform measure over
dimer coverings of $\mathbb T_L$ such that the height changes by
$\lfloor L\rho_i\rfloor$ when winding once in the $\vec e_i$
direction, and $\pi^L_\rho(f)$ tends to $\pi_\rho(f)$ as $L\to\infty$
for every local observable $f$ \cite{KOS}.

Like for the honeycomb lattice, one has
\begin{proposition}
  The measure $\pi^L_\rho$ is stationary.
\end{proposition}
\begin{proof}
The only point where the proof differs w.r.t. the honeycomb lattice case is
the way one shows that $|\Omega_\sigma|:=\sum_b|I^+_b|=\sum_b|I^-_b|=:|\Omega_\sigma^{(-1)}|$, as after 
\eqref{eq:18}. Recall that it is sufficient to show that, after any
elementary move, the difference
$|\Omega_\sigma|-|\Omega_\sigma^{(-1)}|$ is unchanged, whatever the
initial configuration $\sigma$ is.

When an elementary move is performed at a face $f$ in column $\ell$, a
bead $b$ jumps from
an edge $e$ to $e'$ that has a common vertex with $e$. This common
vertex belongs to either $Y_\ell$ or $Y_{\ell-1}$ (recall that $Y_\ell$
is the set of vertices common to columns $\ell,\ell+1$). Assume
w.l.o.g. that the former is the case, as in Figure \ref{fig:millecasi}, and
that $e'$ is higher than $e$ in column $\ell$.
After the move, $|I^+_b|$ decreases by $1$ and $|I^-_b|$ increases by
$1$. On the other hand, it is clear that $|I^\pm_{b'}|$ is unchanged for
beads $b'$ on column $\ell+1$, or on any other column except $\ell-1$. Therefore, we have to find a change $+2$ of
$|\Omega_\sigma|-|\Omega_\sigma^{(-1)}|$ coming from column $\ell-1$.
Call $b^+$, resp. $b^-$, the first bead above (resp. below) $b$ in
column $\ell-1$, and call $b'$ the bead ``between'' $b^+$ and $b^-$ in
column $\ell-2$. (The notion of ordering for beads in neighboring
columns was introduced in Section \ref{sec:estensionedomini}). Then, with reference to Fig. \ref{fig:millecasi},
note that: 
\begin{figure}[h]
  \includegraphics[width=8cm]{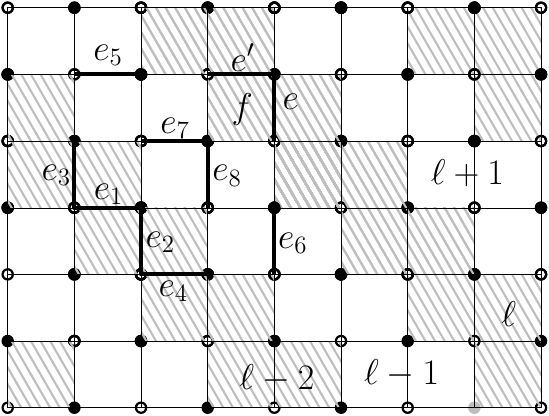}
\caption{}
\label{fig:millecasi}
\end{figure}
\begin{itemize}
\item if $b'$ is at or higher than edge $e_3$, then $b^+$ is at or
  higher than $e_5$ and  $I^-_{b^+}$ is
  the same, irrespectively of whether $b$ is at $e$ or $e'$. On the
  other hand, edges $e_7,e_8$ are accessible to $b^-$ if $b$ is at
  $e'$ and are not if $b^-$ is at $e$, so $|I^+_{b^-}|$ differs by $2$
  in the two cases. Altogether, when $b$ is moved
  from $e$ to $e'$, the contribution of $b^-$ to the change of
  $|\Omega_\sigma|-|\Omega_\sigma^{(-1)}|$ is $+2$, as desired;
\item symmetrically, when $b'$ is at or lower than $e_4$ then $b^-$ is
  at or lower than $e_6$. When $b$ is moved from $e$ to $e'$,
  $I^+_{b^-}$ does not change, while $|I^-_{b^+}|$ decreases by $2$,
  since $e_7,e_8$ are not available positions any more. Again, we get a
  change $+2$ for   $|\Omega_\sigma|-|\Omega_\sigma^{(-1)}|$, this
  time coming from $b^+$.
\item finally, suppose that $b'$ is at $e_1$ or $e_2$. If $b$ is at
  $e$ then position $e_7$ is available for $b^+$ and $e_8$ is not available for
  $b^-$, while if $b$ is at $e'$ the opposite holds. As a consequence,
  both $b^+$ and $b^-$ contribute $+1$ to the change of $|\Omega_\sigma|-|\Omega_\sigma^{(-1)}|$.
\end{itemize}
  \end{proof}

Deducing stationarity of $\pi_\rho$ on the infinite
graph from stationarity of $\pi^L_\rho$ on the torus works exactly the same on
$\mathcal H$ or $\mathbb Z^2$; for definiteness, in Section \ref{sec:stazionarieta} we
will stick to the former case.

\section{The discrete Hammersley dynamics (DHD)}
\label{sec:DHD}
On the way towards Theorem \ref{th:losanghe}, let us switch for a moment to a
one-dimensional interacting particle system known as Discrete Hammersley Dynamics (DHD)
\cite{FM}. The configuration space of the DHD consists of particle
configurations on $\mathbb Z$ (at most one particle per site). Each
site of $\mathbb Z$ has an
i.i.d. Poisson clock of rate $1$. When a clock rings at
a site $x$, if the site is occupied then nothing happens; otherwise, take the first particle to the right of $x$ and move it to $x$. Note
that each particle moves to the left with rate equal to the number $n$
of
empty sites before the next particle to the left, and the new position
is uniform among the $n$ sites. We call 
 $z_t(n)$ the
position of the $n^{th}$ particle ($n\in \mathbb Z$)
 at time $t$.
Particles are labelled in the initial condition in such a way that $z_0(n)<z_{0}(n+1)$, with some
arbitrary choice of whom to label $0$ (for instance, it could be the
first particle to the right of the origin). Labels do not change as
particles move.

The works \cite{AD,Seppa} consider
instead the (continuous) Hammersley process \cite{AD}, which
is defined similarly as the DHD, except that particles live on $\mathbb
R$ instead of $\mathbb Z$: again, each particle moves to the left with
rate equal to the available space before the next particle and the new
position is chosen uniformly in the available interval. In
\cite{Seppa} it is proven (among many other results):
\begin{theorem}
\label{th:DHD}
  If  $ \lim_{n\to-\infty}n^{-2}z_0(n)=0$, then the
dynamics is well defined at all times: the displacement of a particle
with respect to the initial position is almost surely finite at all
finite times.
\end{theorem}

Theorem \ref{th:DHD} extends immediately to the
DHD \cite{FM} and is obtained with the help of a Harris-type
graphical construction, that we recall here.
 To each site of $\mathbb Z$ associate an independent
Poisson point process of density $1$ on $\mathbb R^+$: this is the set
of times when the clock at that site rings. Given a realization of all these
i.i.d. Poisson processes and given $0\le s<t$, $-\infty<a<b<\infty$,
we can consider the set of all possible up-right
paths in the rectangle $(a,b]\times(s,t]$, i.e. sequences
$(x_1,t_1),\dots,(x_n,t_n)$ of space-time points in the point process in the
rectangle, with $x_1<\dots<x_n$ and $t_1<\dots<t_n$. Note that
inequalities are strict (for times this is not restrictive since with
probability one there is at most one clock ringing at a given
time). Let as in  \cite{AD,Seppa}
${\bf L}((a,s),(b,t))$ 
be the maximal number of points of the Poisson processes on  one such path. Let also 
\[
\Gamma((a,s),t,k)=\inf\{h\ge0:{\bf L}((a,s),(a+h,t))\ge k\}.
\]
Then (this is given in \cite{AD,Seppa} in the continuous Hammersley
process where ${\bf L}$ and $\Gamma$ are defined similarly, but the same
holds true also for the DHD) for every $t\ge0$ 
\begin{eqnarray}
  \label{eq:38}
  z_t(n)=\inf_{j\le n}\{z_0(j)+\Gamma((z_0(j),0),t,n-j)\}.
\end{eqnarray}

Note that the DHD has the following
monotonicity property: 
\begin{lemma}[Monotonicity for the DHD]
\label{th:DHDmono}
If we take two initial conditions 
such that $z_0(n)\le z'_0(n)$ for every $n$ and if we let them evolve
using the same Poisson clocks, then the partial order is
preserved at all later  times.
\end{lemma}
\begin{proof}
  This is immediate from \eqref{eq:38}: if some $z_0(j)$ is changed to
  $z_0(j)-a, a\in \mathbb N$, then $\Gamma((z_0(j),0),t,n-j)$
  increases at most by $a$.
\end{proof}
The representation \eqref{eq:38} also allows to get an upper bound on the
probability that the displacement of a particle is large. 
Indeed, if $z_t(n)-z_0(n)\le -k$ then  there
exists $j\le n$ such that
\[
{\bf L}((z_0(j),0),(z_0(j)-k+(z_0(n)-z_0(j)),t))\ge n-j.
\]
With a union bound, the probability (conditionally on the initial
positions) that $z_t(n)-z_0(n)\le -k$ is upper bounded by
\begin{eqnarray}
  \label{eq:15}
\sum_{j< n:z_0(j)-z_0(n)\le -k} P({\bf
  L}((0,0),(0,z_0(n)-z_0(j)-k,t))\ge n-j), 
\end{eqnarray}
where $ P$ denotes the expectation only with respect to the Poisson
clocks.
One has\footnote{see also \cite[Lemma 4.1]{Seppa} that is given for
  the continuous Hammersley process}
\begin{eqnarray}
  \label{eq:28}
 P({\bf L}((0,0),(h,t))\ge k)\le  (th)^k/(k!)^2.
\end{eqnarray}
Indeed, there are $h!/(k!(h-k)!)$ strictly increasing distinct
sequences $0<x_1<\dots x_k\le h$. Given one of these, the probability
that there is an up-right path $(x_1,t_1),\dots,(x_k,t_k)$ equals the
probability that a Poisson random variable of average $t$ equals at
least $k$. 
On the other hand, 
if $X $ is a Poisson variable of
average $t$ then for $k\ge1$
\begin{eqnarray}
  \label{eq:37bis}
 P(X\ge k)=\sum_{n\ge
   k}e^{-t}\frac{t^n}{n!}
=\sum_{m\ge0}e^{-t}\frac{t^{m+k}}{(m+k)!}
  \le \frac{t^k}{k!}.
  \end{eqnarray}
because $(m+k)!\ge k! m!$.
Then, \eqref{eq:28} follows from 
\[
\frac{h!}{(h-k)!}\frac1{k!}\le \frac{h^k}{k!}.
\]
 Let us call $\mathbb P_z$ the law of the DHD
started from an initial configuration $z=\{z_0(n)\}_{n\in \mathbb Z}$. From \eqref{eq:28} we have then
\begin{eqnarray}
  \label{eq:25}
  \mathbb P_z(z_t(n)-z_0(n)\le -k)\le \sum_{j< n:z_0(j)-z_0(n)\le -k}\frac{t^{n-j}(z_0(n)-z_0(j)-k)^{n-j}}{((n-j)!)^2}.
\end{eqnarray}
This bound will be used in Section \ref{sec:propa}.

\section{The process started from $\pi_{\rho}$ is well-defined}
\label{sec:well-d}

Here we prove ``the first (and easier) half'' of Theorem \ref{th:losanghe},
i.e. the bead displacement is finite for almost every initial
condition sampled from $\pi_\rho$.

\begin{proposition}
\label{th:ex}
Suppose that the initial
configuration $\sigma$ is in the set
\begin{eqnarray}
  \label{eq:13}
  Y=\{\sigma: \text{ for every $\ell$, } \lim_{n\to\infty}n^{-2}z_0(\ell,n)=0=\lim_{n\to\infty}n^{-2}z_0(\ell,-n)\}.
\end{eqnarray}
Then the process is well-defined at all times: for every $(\ell,n)$,
almost surely $z_t(\ell,n)-z_0(\ell,n)$ as defined in \eqref{eq:37} is
finite for all $t\ge0$.
\end{proposition}

Actually, when $q=0$ (resp. when $p=0$) the condition
$\lim_{n\to\infty}n^{-2}z_0(\ell,-n)=0$
(resp. $\lim_{n\to\infty}n^{-2}z_0(\ell,n)=0 $) is not necessary. Note also that $\pi_{{\rho}}(Y)=1$
for any non-extremal slope.
Indeed, $z_0(\ell,n)$ is just the sum of the first $n$ inter-bead distances
along  column $\ell$. Since the measure $\pi_{{\rho}}$ is ergodic
for the action of $\mathbb Z^2$, $n^{-1}z_0(\ell,n)$ converges $\pi_{{\rho}}$-almost
surely to the finite limit $1/\rho_3$.
\begin{proof}[Proof of Proposition \ref{th:ex}] Fix some column
  $\ell$. We want to prove that, say,
  $z_t(\ell,n;K)-z_0(\ell,n)$ is almost surely bounded away from minus
  infinity, uniformly in $K$ and for every $n$. Take the DHD slowed down by a factor $q$
  (i.e. its clocks ring
with rate $q$ and not $1$) with initial condition
$z_0(n)=z_0(\ell,n)$ for every $n$ and couple
the DHD and the bead dynamics by establishing that the $q$-clocks (within distance
$K_q$ from the origin)  on
column $\ell$ of the lozenge
dynamics are the same as the corresponding  clocks of the DHD (the DHD
has no $p$-clock
). Then, 
bead positions are 
dominated by those of the  DHD, in the sense that $z_t(n)\le z_t(\ell,n;K)$
for all times and for all $n$. In fact, call
$s_i, i\ge1$ the ordered times when one of the
finitely many  clocks in column $\ell$ of
the dynamics $X_{t;K}$ ring. We have $z_0(\cdot)\le
z_0(\ell,\cdot;K)$ (actually with equality). At time $s_1^-$ the inequality is still true, since the
beads in $\ell$ have not moved while some DHD particles may have moved
to the left. At time $s_1$, one of the following cases occurs:
\begin{itemize}
\item a $p$-clock rings. Then, a bead might move upward and nothing
  happens for the DHD. We have in this case obviously
  \begin{eqnarray}
    \label{eq:50}
z_{s_1^+}(\cdot)\le
z_{s_1^+}(\ell,\cdot;K).
  \end{eqnarray}
\item a $q$-clock rings at an edge $e$ within distance $K_q$ from the origin, but no bead can be moved to
  $e$ without pushing other beads. Again \eqref{eq:50} holds (for the
  DHD, a particle can move to the left).
\item a $q$-clock rings at an edge $e$ within distance $K_q$ from the origin and the bead just above it,
  call $(\ell,n)$ its label, can be moved to $e$. By assumption, for the DHD, the first
  particle at position greater or equal to $e$ has index $m\ge
  n$. After the update, for the DHD one has particle $m$ at $e$ and
  for the lozenge process one has bead $(\ell,n)$ at $e$. All other
  particles/beads are unchanged. Clearly then \eqref{eq:50} holds also
  in this case.
\end{itemize}
The argument is then repeated inductively starting from time $s_1^+$.

Since, by Theorem \ref{th:DHD}, $z_t(n)-z_0(n)>-\infty$ almost surely,
we conclude that $z_t(\ell,n;K)-z_0(\ell,n)$ is almost surely bounded
away from $-\infty$, uniformly in $K$.
\end{proof}

\section{Large gaps and propagation of information}
\label{sec:propa}
Let $B_R$ be the ball of radius $R$ centered at the origin of
$\mathcal H$. 
\begin{definition}
Let $\Delta(R,t)$ be the largest integer $n$ such that there exist horizontal edges 
$e_1\in B_R$ and $e_2$ on the same column of $e_1$, at distance $n$ from
it, such that at time $t$ 
there is 
no bead between them.
Also, let $\Delta(R,\le t)=\sup\{\Delta(R,s),s\le t\}$.
  
\end{definition}

We need a preliminary result, giving an upper bound on the probability
of having a large gap among beads. We start from the case of the torus:
\begin{lemma}
\label{th:deltamaxToro}
  For $\zeta\in \mathbb N$ there exists a constant $C=C_\zeta$ such that, for
  all $T>0, R\ge 1$ and
  $L$ large enough,
  \begin{gather}
    \int \pi^L_{\rho}(d\sigma)\mathbb P_{\sigma}[\Delta(R,\le T)\ge C_\z\log R]\le T\frac{C_\z}{R^{\z}}.
  \end{gather}
\end{lemma}
To be precise, the constant $C_\z$ also depends on the density vector $\rho$ (through the constant $C(\lambda,u,\rho)$ of Lemma \ref{th:DeltaH}); in this  section, for lightness of notation, we often keep the $\rho$ dependence implicit.
\begin{proof}[Proof of Lemma \ref{th:deltamaxToro}]
From Lemma \ref{th:DeltaH} and convergence of $\pi^L_\rho$ to
$\pi_\rho$ it is easy to see that, for $L$ large enough,
\begin{gather}
  \label{eq:76}
  \pi_\rho^L(\Delta(R,0)\ge C_\z\log R)\le  C_\z R^{-\z}
\end{gather}
if $C_\z$ is chosen sufficiently large. Using stationarity of
$\pi^L_\rho$, this holds for every fixed $t\ge0$.
Then, 
\begin{gather}
  \label{eq:77}
 \mathbb E_{\pi^L_\rho}\left(\int_0^{T+1} {\bf 1}_{\{
      \Delta(R,t)\ge C_\z\log R\}} dt\right)\le (T+1)C_\z\, R^{-\z}.
\end{gather}
Let
\begin{gather}
  \tau=\inf\{t>0:\Delta(R,t)>2 C_\z\log R\}
\end{gather}
and observe that, after time $\tau$, a clock has to ring in the ball
$B_{R+C_\z\log R}$ before $\Delta(R,t)$ becomes strictly smaller than
$C_\z\log R$ (this is just a necessary condition: not every ring in
$B_{R+C_\z\log R}$ decreases $\Delta(R,t)$). Note that the realization
of the Poisson clock rings at times $t>\tau$ is independent of the
process up to $\tau$ (and of $\tau$ itself). On the other hand, with
probability $u$ uniformly bounded away from zero, none of the $O(R^2)$
clocks in $B_{R+C_\z\log R}$ rings in the time lag
$[\tau,\tau+1/R^2]$.  In conclusion,
\begin{gather}
  \mathbb E_{\pi^L_\rho}\left(\left.\int_0^{T+1} {\bf 1}_{\{
      \Delta(R,t)\ge C_\z\log R\}} dt\right|\tau\le T\right)\ge u/R^2.
\end{gather}
Together with \eqref{eq:77} we get that
\begin{gather}
  \mathbb P_{\pi^L_\rho}(\tau\le T)\le (T+1)C_\z R^{2-\z}/u.
\end{gather}
We conclude by observing that $\{\tau\le T\}=\{\Delta(R,\le T)> 2 C_\z\log
R\}$ and recalling that $\z$ can be chosen as large as wished.
\end{proof}

For the dynamics on $\mathcal H$ the same argument does not work since
we do not know (yet) that $\pi_\rho$ is stationary. A similar  result however
still holds, but the proof requires a comparison with the DHD we
introduced above:
\begin{lemma}
\label{th:deltamax} For any $T<\infty,\z\in \mathbb N$ there exists a constant
$C=C(\z,T)$ such that for $R$ large
\begin{eqnarray}
  \label{eq:29}
\int \pi_{\rho}(d\sigma)\mathbb P_{\sigma}[\Delta(R,\le T)\ge C\log R]\le \frac{C}{R^{\z}}.
\end{eqnarray}
\end{lemma}
A useful variant of Lemma \ref{th:deltamax} that we will use later (and
whose proof follows almost exactly the same argument) is:
\begin{corollary} \label{th:comodo}
Fix a horizontal edge $e$ and  a time $T>0$.
For $n^\pm\ge 0$ let $A_{T,n^+,n^-}$ be the event that there exists a time $s\le T$ and
horizontal edges $e^\pm$  on the same column as $e$, with $e^+$
at distance $n^+$ above $e$ and $e^-$ at distance $n^-$ below
it, such that at time $s$ there is a bead at $e^\pm$ and no bead
between them. There exists $C=C(\rho,T)$ such that 
  \begin{eqnarray}
    \mathbb P_{\pi_\rho}(A_{T,n^+,n^- })\le C e^{-(1/C)(n^++n^-)}.
  \end{eqnarray}
\end{corollary}

For the proof of Lemma \ref{th:deltamax} we need the following
preliminary result:
\begin{lemma}
  \label{lemma:ndovai} Recall that $\hat \pi_\rho$ is the Gibbs
  measure conditioned to have a bead at $e_0$ and that $\phi_t-\phi_0$
  is the displacement of the tagged bead at time $t$. Then, for
  every $T>0$ there
  exists a positive constant $a=a(\rho,T)$ such that for every $D\ge0$
  \begin{eqnarray}
    \mathbb P_{\hat \pi_\rho}(\exists t\le T:|\phi_t-\phi_0|\ge D)\le a\exp(-D/a).
  \end{eqnarray}
\end{lemma}
\begin{proof}[Proof of Lemma \ref{lemma:ndovai}] To fix ideas let us
  prove that 
 \begin{eqnarray}
    \mathbb P_{\hat \pi_\rho}(\exists t\le T:\phi_t-\phi_0\le - D)\le a\exp(-D/a).
  \end{eqnarray} 
  We have seen in the proof of Proposition \ref{th:ex} that the
  downward displacement of a bead $b$ is at all times stochastically smaller than the
  leftward displacement of a DHD particle (for the DHD with clocks of
  rate $q$) up to the same time, started
  from a configuration where the particles are at the same position as
  the beads in the column corresponding to $b$. Since the DHD
  particles move only to the left, the event
  $\{\exists t\le T:\phi_t-\phi_0\le - D\}$ means that the DHD
  particle corresponding to $b$ has moved more than $D$ by the
  \emph{non-random} time $T$.

Call $n$ the label of the tagged bead $b$ in its column, initially at
position $z_0(n):=e_0$,  and  go back to
\eqref{eq:25}. 
Observe that if $z_0(n)-z_0(n-r)\ge u$ then there are at most $r$
beads in a set of $u$ adjacent horizontal edges below $z_0(n)$.
Using Lemma \ref{th:DeltaH} we see that, except with probability
exponentially small in $D$,
one has 
\begin{gather}
  \label{eq:31}
 z_0(n)-z_0(n-r)<\frac 1{\epsilon(\rho)}\max\left[r,\epsilon(\rho) D
\right]\quad\text{ for every } r\ge 1
\end{gather}
for some positive $\epsilon$
  depending only on the slope $\rho$.
  Then, from \eqref{eq:25},  on the event \eqref{eq:31}
\begin{gather}
  \label{eq:36}
   \mathbb P_z(z_T(n)-z_0(n)\le -D)\le \sum_{r>\epsilon(\rho)D
   }\frac{T^r(r/\epsilon(\rho))^r}{(r!)^2} \le c(\rho,T)\,
   e^{-c'(\rho,T)D \log D}
\end{gather}
that decays super-exponentially in $D$.
\end{proof}

\begin{proof}[Proof of Lemma \ref{th:deltamax}]
  On the event $\Delta(R,\le T)\ge C\log R $ there exists a time $s\le
  T$ and a horizontal edge $e\in B_R$ such that at time $s$ there is no bead in
  the $C\log R$ horizontal edges immediately above or immediately
  below $e$. Assume w.l.o.g. that the former is the case. Let $e^+$
  (resp. $e^-$) be the lowest horizontal edge above $e$ (resp. the
  highest edge below $e$) where there is some bead $b^+$ (resp. $b^-$)
  at time $s$. Call $N\ge C\log R$ the distance between $e^+$ and
  $e^-$.
There are two possible cases:
\begin{itemize}
\item [(i)] at time zero bead $b^+$ is within distance $N/10$ from $e^+$
  and similarly $b^-$ is within distance $N/10$ from $e^-$.  This
  implies that at time zero there is no bead in a vertical interval of
  length $N/2$, centered on the face at distance $N/2$ above $e$.
  Since in the stationary measure the distance between neighboring
  beads has exponential tails (Lemma \ref{th:DeltaH}) and $N\ge C\log
  R$, this event has probability \[O\left(R^2 \exp\left(-a(\rho)\,
      C\log R\right)\right)\] for some positive $a$ depending only on
  the slope $\rho$, where the factor $R^2$ comes from a union bound
  over all possible positions of $e$. Choosing $C=C_\z$ sufficiently
  large, we get a $O(R^{-\z})$ bound.
\item [(ii)] At time zero, either $b^+$ is at distance $n\ge N/10$
  from $e^+$, or $b^-$ is at distance $m\ge N/10$ from $e^-$. Say, to fix ideas,  that
  the former is the case.
This implies that at the (random) time $s\le T$ the bead $b^+$ has moved, say downward,  a distance $n\ge N/10$
with respect to the initial position. 
  Thanks to Proposition \ref{lemma:ndovai}, 
  this has probability exponentially small in $n$. Summing over $n\ge
  N/10$, over the possible values of $N\ge C\log R$ and over the
the $O(R^2)$ possible positions of $e$ gives a bound $O(R^{-\z})$ if $C$ is chosen large enough.
\end{itemize}
\end{proof}

As an application, we show that information does not propagate
instantaneously through the system:
if two initial conditions sampled from
equilibrium differ only outside a ball of radius $R$, it is very
unlikely that in a short time the discrepancy propagates to reach the
center of the ball.
It is useful to give a proof of this fact, since an extremely similar
argument will provide the proof of Theorem \ref{th:losanghe}. For usual short-range
systems one has a ballistic propagation bound: information does not
travel more than a distance $const\times t$ in a time interval $t$ (cf. for instance
\cite[Sec. 3.3]{MartinStF}). The situation is more intricate here due
to the presence of a-priori unbounded gaps
among beads.

\medskip

Let the pair $(\sigma,\sigma')\in\Omega_{\mathcal H}^{\otimes 2}$ be distributed
according to some law $\nu$ such that
$\sigma\sim \pi_\rho$, $\sigma'\sim \pi_\rho$ and $\sigma,\sigma'$
coincide in $B_R$. Couple the two processes  by using the same
Poisson clocks for both and call $\mathbb P_\nu$ the law of the joint
process $(\sigma_t,\sigma'_t)$.
Let $\delta_{e}(t)\in\{0,1\}$ (resp. $\delta'_{e}(t)\in\{0,1\}$) be the bead
occupation variable at time $t$ at a fixed horizontal edge $e$ (say at the
center of $B_R$)  for the process started from $\sigma$
(resp. $\sigma'$). Let also
$\Delta_{max}=\max(\Delta(R,\le T),\Delta'(R,\le T))$, with
$\Delta'(R,\le T)$
referring to the process started from $\sigma'$.
\begin{proposition}
\label{th:storieidentiche}
For every $T<\infty,\z\in \mathbb N$ there is a constant $C$ such that 
\begin{eqnarray}
  \label{eq:24}
  \mathbb P_{\nu}(  (\delta_{e}(t))_{t\in[0,T]}\not\equiv
  (\delta'_{e}(t))_{t\in[0,T]})\le \frac {C}{R^\z}.
\end{eqnarray}
\end{proposition}

\begin{proof}[Proof of Proposition \ref{th:storieidentiche}]

We have from Lemma \ref{th:deltamax} 
\begin{gather}
  \label{eq:14}
  \mathbb P_{\nu}((\delta_{e}(t))_{t\in[0,T]}\not\equiv (\delta'_{e}(t))_{t\in[0,T]})\\\le
\mathbb P_{\nu}\left[(\delta_e(t))_{t\in[0,T]}\not\equiv (\delta'_e(t))_{t\in[0,T]};
  \Delta_{max}\le C_\z \log R\right]+C_\z/R^\z.
\end{gather}
On the event $(\delta_e(t))_{t\in[0,T]}\not\equiv (\delta'_e(t))_{t\in[0,T]}$
call $t_1\le T$ the first time when $\delta_e(t_1)\ne \delta'_e(t_1)$. 
There are two possible cases:
\begin{enumerate}
\item [(i)]  $\delta_e(t_1^-)=\delta'_e(t_1^-)=1$ and say
  $\delta_e(t_1)=0\ne\delta'_e(t_1)=1$. In this case at $t_1$ a clock rings
  in the column $0$ (the one of $e$) at a horizontal edge $x_1$ within distance
  $\Delta_{max}$ from $e$ and in configuration $\sigma'_{t_1^-}$ (but
  not in $\sigma_{t_1^-}$) a bead in a neighboring column is preventing the
  bead at $e$ to move to $x_1$. At time
  $t_1$ there is therefore a horizontal edge $e_1$ in column $\pm1$, with
  distance within $\Delta_{max}+1$ from $x_1$ (the $+1$ is because it
  is on the neighboring column), where the bead
  occupation variable is different.

 \item [(ii)]$\delta_e(t_1^-)=\delta'_e(t_1^-)=0$ and say
  $\delta_e(t_1)=0\ne\delta'_e(t_1)=1$. This means that at $t_1$ the
  clock at $e$ rings (in this case we set $x_1:=e$) and that in
  configuration $\sigma_{t_1^-}$ (but not in $\sigma'_{t_1^-}$) a
  particle in one of the two neighboring columns is preventing a certain
bead
(below $e$ if the clock is a $p$-clock and above the origin $e$ if
it is a $q$-clock) to reach $e$. In particular, as in case
  (i), at time
  $t_1$ there is an edge $e_1$ in column $\pm1$ within distance
  $\Delta_{max}+1$ from $x_1$, where the bead
  occupation variable is different. 

\end{enumerate}
Call $t_2<t_1$ the first time $s$ at which $\delta_{e_1}(s)\ne
\delta'_{e_1}(s)$. On the event $\Delta_{max}\le C_\z \log R$, we have
that $t_2>0$ because $e_1$ is in the ball $B_R$ where initial
conditions coincide.

We iterate the argument  (cf. Fig. \ref{fig:catene}), and as before we deduce that at $t_2$ there
is an edge $x_2$ in the column of $e_1$, within distance $\Delta_{max}$
from it, where a clock rings and an edge $e_2$ in a column neighboring the one of $e_1$, and
at distance within $\Delta_{max}+1$ from $x_2$, where the
bead occupation variable is different. The iteration stops when $e_n$
is outside the ball $B_R$ of radius $R$. Note that $x_i,x_{i+1}$ are
within vertical distance $2(\Delta_{max}+1)\le 3\Delta_{max}$ and
horizontal distance $1$ from each other.

\begin{figure}[h]
  \includegraphics[width=6cm]{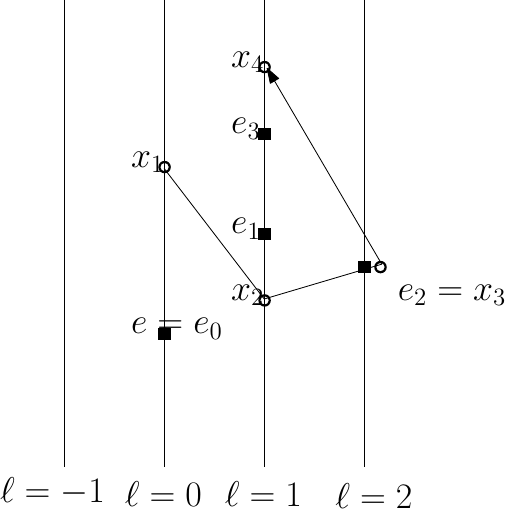}
\caption{An example of the iteration. We have that $|x_i-e_{i-1}|\le
  \Delta_{max}$ and $|e_i-x_{i-1}|\le (\Delta_{max}+1)$. When
  $x_{i+1}=e_{i}$ it means that we are in case (ii) above. The arrow
  follows the chain $x_1,x_2,\dots$}
\label{fig:catene}
\end{figure}

Altogether, if $(\delta_e(t))_{t\in[0,T]}\not\equiv
(\delta'_e(t))_{t\in[0,T]}$ then either $\Delta_{max}\ge C_\z \log R$, or
there exists:
\begin{itemize}
\item a chain of sites $x_1,\dots,x_n$, with $x_i, x_{i+1}$ on
neighboring columns, $x_1$ on the column $\ell=0$ and within distance
$C_\z\log R$ from $e$ (the center of the ball $B_R$), $|x_i-
x_{i+1}|\le 3C_\z\log R$ and
$|x_n|\ge R/2$;
\item a
sequence of times $0\le t_n<t_{n-1}<\dots t_1\le T$ such that either
the $p$-clock or the $q$-clock at $x_i$
rings at time $t_i$.
\end{itemize}
 We get with a union bound
\begin{eqnarray}
  \label{eq:14bis}
  \mathbb P_{\nu}((\delta_e(t))_{t\in[0,T]}\not\equiv (\delta'_e(t))_{t\in[0,T]})\le C_\z/R^\z+
\sum_{n\ge R/(6C_\z\log R)}N_n\,P_n,
\end{eqnarray}
where $P_n$ is the probability that a Poisson
variable of average $T(p+q)$ is at least $n$, while $N_n$ is the number of all
possible distinct chains $x_1,\dots,x_n$  of $n$ sites with the above specified
properties. Of course $N_n\le(C'_\z\log R)^n$ for some constant $C'_\z$ while 
\eqref{eq:37bis} gives
\[
P_n\le e^{-n \log (n/T(p+q)) +n}.
\]
The sum in \eqref{eq:14bis} is $o(R^{-\z})$.
\end{proof}

\begin{remark}
  \label{rem:scambio} Take $\sigma$ sampled from
  $\pi_\rho$ and let $\sigma_t,\sigma'_t$ be the coupled  processes with
  the same Poisson clocks and the same initial condition, except that $\sigma_t$
has cutoff parameter $K=(K_p,K_q)$ and  $\sigma'_t$ has a
  different cutoff parameter $K'=(K'_p,K'_q)$. With the same ideas as
  for Proposition \ref{th:storieidentiche} it is possible to prove
  that
  \begin{gather}
    \label{eq:41}
   \int \pi_\rho(d\sigma)\mathbb P_\sigma((\delta_e(t))_{t\in[0,T]}\not\equiv (\delta'_e(t))_{t\in[0,T]})=\epsilon(K,K')
  \end{gather}
with $\epsilon(K,K')\to0$ when $\min(K_p,K'_p,K_q,K'_q)\to\infty$.
From this one can  deduce that the order how the cutoffs in \eqref{eq:37} are removed is irrelevant.
\end{remark}

\section{Stationarity of Gibbs measures in the infinite graph}
\label{sec:stazionarieta}
We will prove Theorems \ref{th:losanghe} and \ref{th:palm} only for
the honeycomb lattice. As for square lattice, once the result is
proven on the torus (cf. Section \ref{sec:toroquadrato}), the extension to the
infinite system works exactly the same.
\subsection{Proof of Theorem \ref{th:losanghe}}
\label{sec:parte2}
Let us first of all prove \eqref{eq:32} in the case where
$f=\prod_{i=1}^k\delta_{e_i}$, where
 $e_1,\dots, e_k$ are horizontal edges
($k\in\mathbb N$) and
 $\delta_e$  is the indicator
function that there is a dimer at $e$.
Choose $R$ large enough so that all $e_i$ are in the  ball $B_R$ and
say close to its center. Call $\nu_R$
(resp. $\nu^L_R$) the
marginal of $\pi_\rho$ (resp. $\pi^L_\rho$) on $ B_R$ (or, to be
pedantic, on $\mathcal
H\cap B_R$) and let
$\sigma_{B_R} $ be sampled from $\nu_R$ and $\sigma'_{B_R}$ from $ \nu^L_R$. From convergence
of $\pi^L_\rho$ to $\pi_\rho$ as $L\to\infty$, we can choose $L$
sufficiently large and a
coupling of $(\nu_R,\nu^L_R)$ such that $\sigma_{B_R}=\sigma'_{B_R}$ except with
probability $\epsilon_R$ that tends to zero as $R\to\infty$.

Let $\sigma\in
\Omega_{\mathcal H}$ and $\sigma'\in \Omega_{\mathbb  T_L}$ (with
$\Omega_{\mathbb  T_L} $ the set of dimer
coverings of $\mathbb T_L$) be sampled as follows. The restrictions
$(\sigma_{B_R},\sigma'_{B_R})$
to $B_R$ are sampled from $(\nu_R,\nu^L_R)$. Given the realization of
$(\sigma_{B_R},\sigma'_{B_R})$, the configuration
$(\sigma_{\mathcal H\setminus B_R},\sigma'_{\mathbb T_L\setminus B_R})$ outside
$B_R$ are sampled independently: $\sigma_{\mathcal H\setminus B_R}$ from
$\pi_\rho(\cdot|\sigma_{B_R})$ and $\sigma'_{\mathbb T_L\setminus B_R}$ from
$\pi^L_\rho(\cdot|\sigma'_{B_R})$. We have therefore that $\sigma\sim
\pi_\rho,\sigma'\sim \pi^L_\rho$ and they coincide in $B_R$, except
with probability $\epsilon_R$.

Now couple the processes $(\sigma_t)_{t\ge0},(\sigma'_t) _{t\ge0}$ started from
$\sigma,\sigma'$ by establishing that the Poisson clocks in $B_R$ are
the same for the two, while those outside $B_R$ are independent. Proceeding exactly like in the proof of
Proposition \ref{th:storieidentiche} and using both Lemma
\ref{th:deltamaxToro} and
\ref{th:deltamax} to estimate the probability
that $\Delta(R,T)\ge C_\z\log R$ in any of the two processes, one finds
that, except with probability $\epsilon_R+R^{-\z}=\epsilon'_R$, the
bead occupation variables at all edges $e_i,i\le k$ for the two
processes coincide up to time $T$. Therefore, for every $t\le T$,
\begin{gather}
  \mathbb E_{\pi_\rho}(f(X_t))=\mathbb
  E_{\pi^L_\rho}(f(X_t))+\epsilon'_R=
\pi^L_\rho(f)+\epsilon'_R=\pi_\rho(f)+\epsilon'_R+\epsilon''_L
\end{gather}
where we used Proposition \ref{th:invarianzatoro}
(stationarity on the torus) in the second equality.
Arbitrariness of $T<\infty$ and of $R$ proves \eqref{eq:32} in the
particular case $f=\prod_{i=1}^k\delta_{e_i}$ (the larger $T $ is, the
larger we have to choose $R$ and therefore $L$).
\medskip

When $f$ is any bounded local function depending only on the
configuration of the horizontal dimers, it is always possible to write
$f$ as a finite linear combination of functions of the form
$\prod_{i=1}^k\delta_{e_i} $, so the claim of the theorem holds also in this
case.

\smallskip

Finally, it remains to consider the case where $f$ is a local function
depending also on the configuration of non-horizontal edges. This
requires a slightly different argument.

Let us start with a simple observation, see
Fig. \ref{fig:determinato}.
\begin{figure}[h]
  \includegraphics[width=3cm]{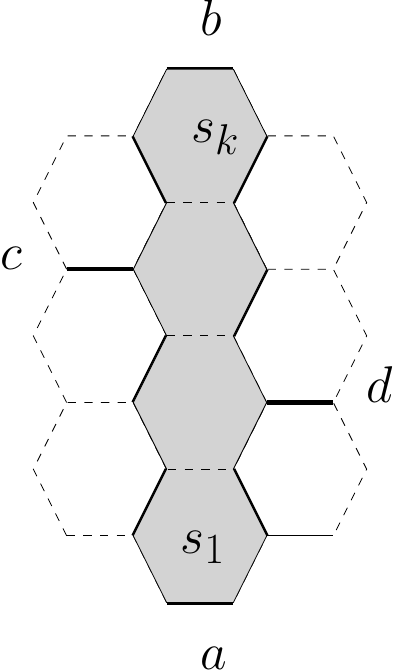}
\caption{Given the horizontal dimers at $a,b,c,d$, all the dimer occupation variables
  at edges of hexagons $s_1,\dots,s_k$ is determined. In fact, all
  north-east oriented edges between $a$ and  $c$ and between  $d$ and
  $b$ are occupied by dimers, and the same is for
  south-east  edges between  $c$ and $b$ or between $a$ and $d$.}
\label{fig:determinato}
\end{figure}
Let $a,b$ be two horizontal edges in the same column $\ell$ and let
$s_1,\dots,s_k$ be the hexagons of $\ell$ included between $a,b$. If we know that the only
beads in $s_1,\dots s_k$ are at
$a,b$ and if we also know the location of the two beads, one in column
$\ell+1$ and one in column $\ell-1$, whose vertical coordinate is between
that of $a$  and of $b$, then we can reconstruct unambiguously the
dimer 
occupation variables of all edges (not just horizontal ones) of hexagons $s_1,\dots,s_k$.

Call $S$ a finite collection of hexagons such that the union of
their edges contains the support of $f$.
Let $\Lambda_K$ be the collection of hexagons that are at graph
distance (on $\mathcal H^*$) at most $K$ from $S$ ($S$ itself
is a subset of $\Lambda_K$). Let $E_K$ be the
event that, for every $s\in S$, there are two beads in $\Lambda_K$, one
below $s$ and one above it. 
From 
the discussion above 
we know that, on the event $ E_K$, the dimer configuration on all the
hexagons in 
$S$ is uniquely identified by
$\eta|_{\Lambda_K}$, the bead configuration in $K$.
Let 
\begin{eqnarray}
  \label{eq:17}
  g(\eta):=g(\eta|_{\Lambda_K})=f(\sigma)1_{E_{\Lambda_K}}
\end{eqnarray}
which depends only on $\eta|_{\Lambda_K}$. 
Let also $\tilde E_K(t)$ the event that
$E_K$ is realized at every $s\le t$. 
We have 
\begin{gather}
  \label{eq:34bis}
  \mathbb E _{\pi_{{\rho}}}(f(X_t))= \mathbb E _{\pi_{{\rho
        }}}(g(\eta_t);1_{\tilde E_K(t)})
  +
O( \|f\|_\infty\mathbb P_{\pi_{\rho}}(\tilde E_{K}(t)^c) )\\=
\mathbb E _{\pi_{{\rho
        }}}(g(\eta_t))  +
O( \|f\|_\infty\mathbb P_{\pi_{\rho}}(\tilde E_{K}(t)^c) ).
\end{gather}
From Corollary 
\ref{th:comodo} we deduce easily that $\epsilon_{K,t}:=\mathbb P_{\pi_{\rho
    }}(\tilde E_{K}(t)^c)$ tends to zero as $K\to\infty$, for every
fixed $t$.
Therefore, 
\begin{gather}
  \label{eq:58}
  \mathbb E _{\pi_{{\rho}}}(f(X_t))= \mathbb E_{{\pi_{\rho}}}(g(\eta_t))+O(\|f\|_\infty\epsilon_{K,t})\\={\pi_{\rho}}(g)+O(\|f\|_\infty\epsilon_{K,t})=\pi_{\rho}(f)+O(\|f\|_\infty\epsilon_{K,t})
\end{gather}
where we used invariance of the Gibbs measure for functions of the
bead configuration in the second equality
and \[\pi_{{\rho}}(g)=\pi_{{\rho}}(f)+
O(\|f\|_\infty\epsilon_{K,0})\] in the last (note $\epsilon_{K,0}\le \epsilon_{K,t}$).
We conclude by letting 
$K\to\infty$.

\subsection{ Proof of Proposition
\ref{th:palm}}
\label{sec:palm}
This is very  similar to the proof of Theorem \ref{th:losanghe}, so we will
be 
very sketchy.  Given any $R>0$ and $\epsilon>0$ one can choose  $L$ sufficiently large
so that there is a probability law for the random couple
$(\sigma,\sigma')\in\Omega_{\mathcal H}\times \Omega_{\mathbb
  T_L}$ such that $\sigma\sim \hat
\pi_\rho,\sigma'\sim\hat\pi^L_\rho$ and $\sigma,\sigma'$ coincide,
except with probability $\epsilon$, in the ball $B_R$. This is done
like at the beginning of the proof of  Theorem \ref{th:losanghe}: in
fact, the total variation distance between the marginals on $B_R$ of
$\hat\pi_\rho,\hat\pi^L_\rho$ tends to zero as $L\to\infty$
(this is because the statement is true for the measures
$\pi_\rho,\pi^L_\rho$ not conditioned to have a bead at the edge
$e_0$, and the probability to have a bead at $e_0$ is
uniformly bounded away from zero). As in Theorem \ref{th:losanghe},
given any $a>0$,
the coupled bead processes $(\sigma_t,\sigma'_t)$ that use the same clocks
in $B_R$ coincide up to time $T$ in the ball $B_a$, except with
probability $\epsilon$, provided $R$ is larger than some
$R_0(a,T)$. On the other hand, by comparing the displacement of a bead
with that of a DHD particle, we see that if $a$ is sufficiently large
(depending only on $T$)
the ``tagged bead'' stays within distance $a/2$ from its
initial position up to time $T$, except
with probability $\epsilon$. In conclusion, the processes
$\hat X_t,\hat X_t'$  re-centered at the position of the tagged bead
of $\sigma_t,\sigma'_t$ coincide (except with probability $2\epsilon$) up
to time $T$ in a ball of radius $a/2$ centered at the origin. Together
with the fact that the re-centered process $\hat X_t'$ has law $\hat
\pi^L_\rho $ at all times (Proposition \ref{th:palmtoro}) this implies
the claim.

\section{Speed and fluctuations}
\label{sec:velfluct}

Here we prove Theorem \ref{th:vvar} and \ref{th:32} about 
average speed and fluctuations of the growth process.

Let $\Lambda $ be the $\ell\times \ell$ box in $\mathcal H$
defined as 
the collection of hexagons obtained by translating a fixed hexagonal
face $x$ (say, the one at the origin of $\mathcal H$) by $a\vec
e_1+b \vec e_2, 0\le a,b\le \ell$. 
Let 
\begin{gather}
Q_\Lambda(t)=\sum_{x\in \Lambda}(h_x(t)-h_x(0))=\sum_{x\in \Lambda}Q_x(t).  
\end{gather}
Remark that
\begin{gather}\label{eq:intuitiva}
  Q_\Lambda(t+\delta)= Q_\Lambda(t)-\sum_e y^{(p)}_e |V(e,\uparrow)\cap
  \Lambda| +\sum_e y^{(q)}_e|V(e,\downarrow)\cap
  \Lambda| +R_{\Lambda,t,\delta}
\end{gather}
with $y^{(p)}_e/ y^{(q)}_e$ the indicator that the $p/q$-clock
 at $e$ rings once in the time interval $[t,t+\delta]$, while the ``error
 term'' $R_{\Lambda,t,\delta}$ includes the contribution to the change of $Q_\Lambda$
 from the events where there are $n\ge 2$ edges $e_1,\dots,e_n$ where
clocks ring in the time interval $[t,t+\delta]$ and where, for every
$i\le n$, 
either $|V(e_i,\downarrow)\cap
  \Lambda|\ne 0$ or $|V(e_i,\uparrow)\cap
  \Lambda|\ne 0$.

\begin{proof}[Proof of \eqref{eq:Qx}]
We want to see that
\begin{gather}
\label{eq:correntemedia}
  \mathbb E_{\pi_\rho}[Q_\Lambda(t)]=(q-p)t \ell^2 J
\end{gather}
with $J$ defined in \eqref{eq:J}. By linearity we can assume $\ell=1$,
i.e. $\Lambda=\{x\}$.

To see that $R_{\{x\},t,\delta}$ can be neglected for $\delta\to0$ let $b^\pm(t)$ be the lowest/highest bead above/below
$x$ in the same column and let $I(t)$ be the collection of horizontal
edges included between $b^-(t)$ and $b^+(t)$.  Let also
$I(t,\delta)=\cup_{s\in[t,t+\delta]}I(s)$.
Then, observe that the only clock rings that can contribute $\pm 1$ to
$Q_x(t+\delta)- Q_x(t)$ necessarily occur in
$I(t,\delta)$.
Then, 
\begin{eqnarray}
  |R_{\{x\},t,\delta}|\le N\, 1_{\{N\ge 2\}}
\end{eqnarray}
where $N$ is a Poisson variable of average $\delta(p+q)  |I(t,\delta)|$.
 Note that the law of
$|I(t,\delta)|$ for the stationary process of
law $\mathbb P_{\pi_\rho}$  is
independent of $t$ and that, 
from Corollary \ref{th:comodo}, the random variable
$|I(0,\delta)|$ has exponential tails. Therefore, $\mathbb
E_{\pi_\rho}|R_{\{x\},t,\delta}|=O(\delta^2)$ and we see 
that
\begin{eqnarray}
  \frac d{dt}\mathbb E_{\pi_\rho}[Q_x(t)]=-p\,\pi_\rho(|\{e:x\in V(e,\uparrow)\}|)+\,q\,
\pi_\rho(|\{e: x\in V(e,\downarrow)\}|)\\=
 (q-p) \pi_\rho(|\{e:x\in V(e,\uparrow)\}|)
\end{eqnarray}
where we used 
 stationarity of $\pi_\rho$ and in the last step its invariance  by
reflections through any hexagon.
\end{proof}

\begin{proof}[Proof of \eqref{eq:varianzaQ} and \eqref{eq:varianzaQ1}]
We compute the variance of $Q_\Lambda(t)$. 
We have (recall \eqref{eq:intuitiva}, where again we can see that
$R_{\Lambda,t,\delta}\approx \delta^2$
for $\delta$ small with the same argument as above), letting for lightness of
notation $\langle\cdot\rangle:=\mathbb E_{\pi_\rho}$,
\begin{gather}
\label{eq:Q21}
  \frac{d}{dt}\mathbb \langle Q_\Lambda(t)^2\rangle=
2 \langle Q_\Lambda(t)K_1(\sigma_t)\rangle+\langle K_2(\sigma_t)\rangle
\end{gather}
where
\begin{gather}
  K_n(\sigma)=(-1)^np\sum_e|V(e,\uparrow)\cap \Lambda|^n+q\sum_e|V(e,\downarrow)\cap \Lambda|^n
\end{gather}
and the sums run over all horizontal edges of $\mathcal H$.
We have  then, recalling also \eqref{eq:correntemedia}, $\langle
K_2(\sigma_t)
\rangle=\pi_\rho(
K_2)
$ and $\partial_t \langle Q_\Lambda(t)\rangle=\pi_\rho(K_1)$,
\begin{gather}
  \label{eq:Mn}
  \frac d{dt}M_2(t):=\frac d{dt}\langle(Q_\Lambda(t)-\langle
  Q_\Lambda(t)\rangle)^2\rangle=
2\langle
(Q_\Lambda(t)-\langle
  Q_\Lambda(t)\rangle)(K_1(\sigma_t)-\pi_\rho(K_1))
\rangle\\+
\pi_\rho(
K_2
).
\end{gather}

One has (see Appendix \ref{app:stime}) that 
\begin{gather}
\label{tec1}
\sup_\ell \ell^{-2}\pi_\rho(|K_n|)\le C_1=C_1(n)<\infty
\end{gather}
and, for every $\delta>0$,
\begin{gather}
 \label{agaa}
\sup_\ell\frac1{\ell^{2+\delta}}\pi_\rho[(K_1
    -\pi_\rho(K_1))^{2}]\le C_2(\delta)<\infty.
\end{gather}
\begin{remark}
\label{rem:sesolo}
  It is likely that the variance of $K_1$ is
  actually of order $\ell^{2}$, without any spurious
  correction. Indeed it is proven in \cite{Boutillier} that, if $f$ is a local
  dimer function and $f_x$ is $f$ translated by $x\in \mathbb Z^2$,
  then $(1/\ell)\sum_{|x|\le \ell}[f_x-\pi_\rho(f_x)]$ satisfies a CLT with
  finite variance. The problem with $K_1$ is that
  $|V(e,\uparrow)\cap \Lambda|,|V(e,\downarrow) \cap \Lambda|$ are not local functions. While
  in principle they are ``almost-local'' (the probability that they
  involve more than $n$ dimers decays at least exponentially in $n$,
  see Lemma \ref{th:DeltaH}), 
  even proving the weaker \eqref{agaa} requires some non-trivial work.
\end{remark}

We have from \eqref{eq:Mn}, from  stationarity and from \eqref{tec1}, \eqref{agaa}
\begin{gather}
\label{eqdiff1}
  \frac{d}{dt}M_2(t)\le
  \sqrt{M_2(t)}\sqrt{\pi_\rho[(K_1(\sigma_t)-\pi_\rho(K_1))^2]}+2\ell^2C_1\\
\le C_3(\delta)\ell^{1+\delta/2}\sqrt{M_2(t)}+2\ell^2C_1
\end{gather}
from which it is then immediate to deduce that
\begin{gather}
\label{eq:M2t}
  M_2(t=\ell)\le C_4\ell^{4+\delta}.
\end{gather}
Now we are ready to prove \eqref{eq:varianzaQ}. Let $x_0$ be a face in
$\Lambda$.
Write 
\begin{gather}
\label{rep1}
  \mathbb P_{\pi_\rho}(|Q_{x_0}(\ell)-\langle Q_{x_0}(\ell)\rangle|\ge
\ell^{2\delta}
  )\\\le 
 \mathbb P_{\pi_\rho}(|Q_{x_0}(\ell)-\langle Q_{x_0}(\ell)\rangle|\ge 
 \ell^{2\delta};
|Q_\Lambda(\ell)-  \langle Q_\Lambda(\ell)\rangle|\le \ell^{2+\delta}
)+
o(1)
\end{gather}
where we used \eqref{eq:M2t} to neglect the event that $|Q_\Lambda(\ell)-  \langle Q_\Lambda(\ell)\rangle|\ge \ell^{2+\delta}$.
On the other hand
\begin{gather}
  Q_\Lambda(\ell)-  \langle Q_\Lambda(\ell)\rangle=-A_1+A_2+A_3\\:=-\sum_{x\in\Lambda}[h_x(0)-h_{x_0}(0)-\pi_\rho(h_x-h_{x_0})]\\
+
\sum_{x\in \Lambda}[h_x(\ell)-h_{x_0}(\ell)-\pi_\rho(h_x-h_{x_0})]+\ell^2[Q_{x_0}(\ell)-\langle Q_{x_0}(\ell)\rangle].
\end{gather}
We have (see again Appendix \ref{app:stime})
\begin{gather}
 \label{tec3}
  \pi_\rho\left[A_1^2\right]=O(\ell^4\log \ell)
\end{gather}
so that, using stationarity and Tchebyshev,
\begin{gather}
  \label{eventicchio} |A_1|,|A_2|\le 
\ell^{2}\log \ell
  ,
\end{gather}
with probability $1-o(1)$.
Finally, we note that if event \eqref{eventicchio} holds and 
at the same time
$|Q_\Lambda(\ell)-  \langle
Q_\Lambda(\ell)\rangle|\le
\ell^{2+\delta}
 $, one cannot
have $|Q_{x_0}(\ell)-\langle Q_{x_0}(\ell)\rangle|\ge
\ell^{2\delta}
$. Eq. \eqref{eq:varianzaQ} is then proven (just let $\ell:=t$).

\begin{remark}
\label{rem:ciro2}
  If for a given slope $\rho$ the condition \eqref{eq:ciro} is satisfied, then one can 
  prove (cf. Remark \ref{rem:ciro} below) that
  $\pi_\rho[(K_1-\pi_\rho(K_1))^2]=O(\ell^2\log \ell)$, to be compared
  with \eqref{agaa}.
Going back to \eqref{eqdiff1} one sees that \eqref{eq:M2t} is then
improved to $M_2(t=\ell)\le C_4 \ell^4\log \ell$. Repeating the
argument that starts with \eqref{rep1}, one sees immediately that
\eqref{eq:varianzaQ} is improved into \eqref{eq:varianzaQ1}.
\end{remark}
\end{proof}

\appendix
\numberwithin{equation}{section}
\section{Some equilibrium estimates}
Here we give upper bounds on the probability that, at equilibrium,
there is a large gap between two consecutive beads in the same
column. We use this information to deduce several useful
equilibrium estimates.

\label{app:stime}

Let $J_r$ be a set of $r$ \emph{adjacent} horizontal edges in the
same vertical column of $\mathcal H$ and $N_r$ the number of beads in $J_r$.
\begin{lemma}
\label{th:DeltaH} Let $\rho$ be a non-extremal slope. For every
$\lambda>0$ and $u>0$ there exists $C=C(\lambda,u,\rho)<\infty$ such
that, 
   for every
   $r\in \mathbb N$,
\begin{eqnarray}
  \label{eq:39}
    \pi_{\rho}(|N_r-\rho_3 r|\ge u r)\le C\exp(-\lambda u r).
\end{eqnarray}
\end{lemma}
Recall that $\pi_\rho(N_r)=\rho_3r$.
\begin{proof}
It is known (cf. \cite[Sec. 6.3]{Klecturenotes}) that $N_r$ is
distributed like the sum of $r$ independent but not identically
distributed Bernoulli random variables $B_i,i\le r$ of parameter $q_i$
satisfying
$\sum_i q_i=r\rho_3$ and $\sum_i q_i(1-q_i)\sim (1/\pi^2)\log r$
as $r\to\infty$.
One has then
\begin{gather}
  \pi_\rho(N_r-\rho_3 r\ge u r)=P(\sum_{i\le r}(B_i-q_i)\ge u r)\\\le 
\exp(-\lambda u r)\prod_i
\left[q_ie^{\lambda(1-q_i)}+(1-q_i)e^{-\lambda q_i}\right].
\end{gather}
Since for every $\lambda>0$ there exists $C_1=C_1(\lambda)$ such that
$\exp(x)\le 1+x+C_1x^2$ for every $x\in [-\lambda,\lambda]$, we get
for $r\ge r_0(\rho)$ large
\begin{gather}
    \pi_\rho(N_r-\rho_3 r\ge u r)\le e^{-\lambda u r}\prod_i[1+C_2(\lambda)
    q_i(1-q_i)]\\\le e^{-\lambda u r+C_2(\lambda)\sum_i q_i(1-q_i)}\le e^{-\lambda
      u r+C_4(\lambda)\log r}\le C(\lambda,u) e^{-\lambda u r/2}.
\end{gather}
The claim is immediately extended to $r\le r_0(\rho)$, possibly
changing $C$ to a new constant $C(\lambda,u,\rho)$.
With a similar argument one estimates $ \pi_\rho(N_r-\rho_3 r\le -u r)$.
\end{proof}

\begin{proof}[Proof of \eqref{tec1}]
  Just note that 
  $|V(e,\uparrow)\cap \Lambda|=k$ implies that there are $k$ hexagons just below
  $e$ with no beads (or $k+d$ of them, if $e$ is at distance $d$ from
  $\Lambda$), an event that has probability exponentially small in
  $k$ (or $k+d$) thanks to Lemma \ref{th:DeltaH}. The average of $K_n$ is then immediately seen to be of
  order $\ell^{2}$. 
\end{proof}
\begin{proof}[Proof of \eqref{tec3}]
  It is well known \cite{KOS} that
the variance of $h_x-h_y$ under $\pi_\rho$ grows like the logarithm of
$|x-y|$. Then, a Cauchy-Schwarz inequality implies the desired estimate.
\end{proof}
\begin{proof}[Proof of \eqref{agaa}] 
By Jensen's inequality and symmetry it suffices to show that the
variance of
  \[f_\Lambda=\sum_e|V(e,\uparrow)\cap\Lambda|\] is $O(\ell^{2+\delta})$.
Write
\begin{gather}
  f_\Lambda=f^{(1)}_\Lambda+f^{(2)}_\Lambda -f^{(3)}_\Lambda
  :=\sum_{e\in\Lambda}|V(e,\uparrow)|
  +\sum_{e\not\in\Lambda}|V(e,\uparrow)\cap\Lambda| -\sum_{e\in\Lambda}|V(e,\uparrow)\setminus\Lambda|
\end{gather}
and, again by Jensen, it is enough to estimate the variance of each of the three
terms. This is easy for $f^{(2)}_\Lambda$  and
 $f^{(3)}_\Lambda$. Indeed, for instance if $e$ is outside $\Lambda$
and at distance $d_e$ from it,  then
$|V(e,\uparrow)\cap\Lambda|\ge n$ implies that there is a sequence of
at least $n+d_e$ adjacent hexagons starting from $e$, where no bead is
present. This has probability exponentially small in $d_e+n$. As a
consequence, if $e_1,e_2\not\in\Lambda$ then 
\begin{gather}
\pi_\rho[|V(e_1,\uparrow)\cap\Lambda|\times|V(e_2,\uparrow)\cap\Lambda|]\le
c \exp(-c'(d_{e_1}+ d_{e_2}))
\end{gather}
from which a bound $O(\ell^{2})$ on the second moment (and therefore
the variance) of $f^{(2)}_\Lambda$ easily follows.  A similar argument
works for $f^{(3)}_\Lambda$.

The case of $f^{(1)}_\Lambda$ is much more subtle. Observe  (cf. Fig.  \ref{fig:determinato2}) that having $|V(e,\uparrow)|=n>0$ is
equivalent to the 
following: the horizontal edge $e^-_n$ that is at distance $n$ below $e$
is occupied by a dimer, and so are the $n$ edges $e^l_1,\dots,e^l_n$
and $e^r_1,\dots,e^r_n$, i.e.
\begin{gather}
  \label{eq:ie}
1_{|V(e,\uparrow)|=n}=\delta_{e^-_n}\prod_{j=1}^n(\delta_{e^l_j}\delta_{e^r_j}).
\end{gather}
\begin{figure}[h]
  \includegraphics[width=2.3cm]{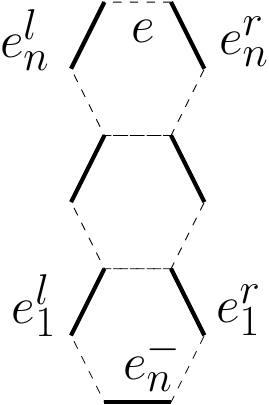}
\caption{The event $V(e,\uparrow)=n$ for $n=3$.}
\label{fig:determinato2}
\end{figure}

We have
\begin{gather}
  f_\Lambda^{(1)}-\pi_\rho(  f_\Lambda^{(1)})=\sum_{n>0}n(f_n-\pi_\rho(f_n))
\end{gather}
with 
$
f_{n}=\sum_{e\in\Lambda} 1_{|V(e,\uparrow)|=n}.
$
We use then Jensen's inequality, $(\sum_i t_i a_i)^2\le \sum_i
t_i a_i^2$ if $t_i\ge0,\sum_it_i=1$, to get
\begin{gather}
\label{eq:usatoJensen}
 \pi_\rho  [(f_\Lambda^{(1)}-\pi_\rho(  f_\Lambda^{(1)}))^2] \le C\sum_{n>0}n^4 \pi_\rho  [(f_{n}-\pi_\rho(  f_{n}))^2] 
\end{gather}
(we chose $t_{n}=1/(C n^2)$ with $C=\sum_{n>0}n^{-2}$).
It remains to estimate the variance of  $f_{n}$. Letting
$V(e;n)=1_{|V(e,\uparrow)|=n}-\pi_\rho(|V(e,\uparrow)|=n)$, write
\begin{gather}
\label{eq:riveu}
  \pi_\rho  [(f_{n}-\pi_\rho(  f_{n}))^2] \le
  \sum_{e,e'\in\Lambda}
|\pi_\rho[ V(e;n)V(e';n)]|.
\end{gather}
Since the event $|V(e,\uparrow)|=n$ implies that there are $n-1$
adjacent hexagons
without beads under $e$, we have from  Lemma \ref{th:DeltaH}, for any $\lambda>0,\delta>0$,
\[
|\pi_\rho[ V(e;n)V(e';n)]|\le C(\lambda,\delta,\rho)e^{-n\lambda/\delta}.
\]
Together with \eqref{eq:riveu} this gives
\begin{gather}
\label{eq:sceglilambda}
   \pi_\rho  [(f_{n}-\pi_\rho(  f_{n}))^2] \le C'(\lambda,\delta,\rho)
  \sum_{e,e'\in\Lambda}|\pi_\rho[
  V(e;n)V(e';n)]|^{1-\delta}e^{-\lambda n},
\end{gather}
where the constant $\lambda$ will be chosen later.

Recall from \eqref{eq:ie} that $1_{|V(e,\uparrow)|=n}$  is a product of dimer indicator functions
on a certain set of $p=2n+1$ (not all horizontal)
edges of $\mathcal H$. Call
$e_1=(\rb_1,\rw_1),\dots,e_p=(\rb_p,\rw_p)$ such edges and
let $e_{p+1},\dots,e_{2p}$ be the analogous edges corresponding to
$1_{|V(e',\uparrow)|=n}$
(of course $e_{i+p}$ is just $e_i$ translated by $e'-e$).
Now we use formula \eqref{eq:multidimero}:
\begin{gather}
\label{eq:VV}
  \pi_\rho[
  V(e;n)V(e';n)]=(k_3)^2(k_1 k_2)^{2n}\widetilde\det
  (K^{-1}(\rw_i,\rb_j))_{1\le
    i,j\le 2p}
\end{gather}
where $\widetilde\det$ means that, since the variables $V(e;n)$ are
centered, 
when we expand the determinant in
permutations $\sigma$  of $\{1,\dots,2p\}$ we have to keep only the permutations such that in the
product there are
$N\ge 2$ ``special'' terms of the type $K^{-1}(\rw_i,\rb_{\sigma(i)})$
with $i\le p$ and $\sigma(i)>p$ or viceversa (note
$N$ is always even). 
Thanks to \eqref{eq:64}, each of the $N$ special terms 
is of order $1/|e-e'|$ for $|e-e'|$ large. We will consider
therefore only the contribution of permutations such that $N=2$ (those
with $N>2$ will give a sub-dominant correction when the sum over
$e,e'\in\Lambda$ is performed; we skip details). W.l.o.g.
we assume that the  special terms are $K^{-1}(\rw_{i_1},\rb_{\sigma(i_1)})$ and
$K^{-1}(\rw_{i_2},\rb_{\sigma(i_2)})$ with $i_1,\sigma(i_2)\le
p, i_2,\sigma(i_1)>p$ (one has afterwards to sum over the $O(p^4)=O(n^4)$
possible choices of $i_1,i_2,\sigma(i_1),\sigma(i_2)$).

The contribution to $\widetilde\det
  (K^{-1}(\rw_i,\rb_j))_{1\le
    i,j\le 2p}$ from such permutations is
  \begin{gather}
\epsilon_{i_1,i_2,\sigma(i_1),\sigma(i_2)}\det
(K^{-1}(\rw_i,\rb_j))_{\{1\le i,j\le p, i\ne i_1, j\ne \sigma(i_2)\}}
\\\times\det
(K^{-1}(\rw_i,\rb_j))_{\{p+1\le i,j\le 2p, i\ne i_2 , j\ne \sigma(i_1)\}}
  \end{gather}
with $\epsilon=\pm 1$ a sign that will play no role later.
We claim that there exists $C(\rho)<\infty$ such that 
\begin{gather}
  \label{eq:GH}
  \det(K^{-1}(\rw_i,\rb_j))_{i\in I,j\in J}\le C(\rho)^r
\end{gather}
for any $r\ge 1$ and sets $I,J$ of cardinality $r$.
If this is the case, from \eqref{eq:VV} we have
\begin{gather}
\label{eq:numerica}
| \pi_\rho[
  V(e;n)V(e';n)]|\le 
  \frac{c(\rho)}{|e-e'|^2}[k_1\,k_2\,C(\rho)^2]^{2n}
n^4\le  \frac{(C')^n}{|e-e'|^2}
\end{gather}
(recall that $n^4$ comes from the summation over the possible values
of $i_1,i_2,\sigma(i_1),\sigma(i_2)$).
Plugging into
\eqref{eq:sceglilambda} and choosing $\lambda$ sufficiently large we get
\begin{gather}
   \pi_\rho  [(f_{n}-\pi_\rho(  f_{n}))^2] \le
   C''(\lambda,\delta)e^{-(n/2)\lambda}\sum_{e,e'\in\Lambda}\frac1{|e-e'|^{2(1-\delta)}}\le 
   C'''(\lambda,\delta)e^{-(n/2)\lambda}\ell^{2+2\delta}.
\end{gather}
Using this estimate in \eqref{eq:usatoJensen} we finally get
\begin{gather}
   \pi_\rho  [(f_\Lambda^{(1)}-\pi_\rho(  f_\Lambda^{(1)}))^2]\le C''(\delta)\ell^{2+2\delta}
\end{gather}
as desired. The contribution from permutations with $N>2$ gives
instead $O(\ell^2)$ since $|e-e'|^{-2}$ is replaced by $|e-e'|^{-N}$
that is summable over $e'\in \mathcal H$. 

It remains to prove \eqref{eq:GH}. This is based on Gram-Hadamard type
bounds (cf. for instance \cite[App. A4]{GM}):
if $f_i,g_i,i\le m$ are vectors in a Hilbert space and $\|\cdot\|$ is
the norm induced by the scalar product $\langle\cdot,\cdot\rangle$, then
\begin{gather}
\label{eq:GH2}
 |\det(\langle f_i,g_j\rangle)_{i,j\le m}|\le \prod_{j\le m}\|f_j\|\,\|g_j\|.
\end{gather}
The second observation (this trick is often used in
constructive Quantum Field Theory, see again \cite[App. A4]{GM}) is
that one can rewrite \eqref{eq:K-1} as 
\begin{gather}
  K^{-1}(\rw_x,\rb_{x'})=\frac1{(2\pi)^2}\int_0^{2\pi}d\theta\int_0^{2\pi}d\phi
  \frac{e^{-i\theta(x'_2-x_2)+i\phi(x'_1-x_1)}}{\tilde
    P(\theta,\phi)}\\=\sum_{y\in\mathbb Z^2}\overline{
    A_x(y)}B_{x'}(y)=:\langle A_x,B_{x'}\rangle
\end{gather}
where $\tilde P(\theta,\phi)=P(e^{i\theta},e^{i\phi})$, $\overline z$ is the complex conjugate of a complex number $z$ and 
\begin{gather}
  A_x(y)=\frac1{(2\pi)^2}\int_0^{2\pi}d\theta\int_0^{2\pi}d\phi\frac{e^{-i\theta(x_2-y_2)+i\phi(x_1-y_1)}}{
\sqrt{|\tilde P(\theta,\phi)|}}
\\
  B_{x'}(y)=\frac1{(2\pi)^2}\int_0^{2\pi}d\theta\int_0^{2\pi}d\phi\frac{e^{-i\theta(x'_2-y_2)+i\phi(x'_1-y_1)}}{
|\tilde P(\theta,\phi)|^{3/2}}
\overline{\tilde P(\theta,\phi)}.
\end{gather}
Finally one applies \eqref{eq:GH} together with the observation that
$\|A_x(\cdot)\|,\|B_x(\cdot)\|$ are upper bounded by a 
constant. Indeed, 
\begin{gather}
\label{eq:speranum}
  \|A_x(\cdot)\|^2=\|B_x(\cdot)\|^2=\frac1{(2\pi)^2}\int_0^{2\pi}d\theta\int_0^{2\pi}d\phi\frac1{|\tilde P(\theta,\phi)|}=:C(\rho)
\end{gather}
which is finite since $\tilde P$ has only simple poles on the torus. 
\end{proof}
\begin{remark}
\label{rem:ciro}
  For a given choice of $\rho$ (and therefore of $k_1,k_2,k_3$)
it may happen that $C(\rho)$ in \eqref{eq:speranum} satisfies
$k_1 k_2 C(\rho)^2<1$. In this case, from the first inequality in 
\eqref{eq:numerica} together with  \eqref{eq:riveu} and
\eqref{eq:usatoJensen} we see that
\begin{gather}
   \pi_\rho  [(f_\Lambda^{(1)}-\pi_\rho(  f_\Lambda^{(1)}))^2]
   =O\left(
\sum_{e,e'\in\Lambda}\frac1{|e-e'|^2}\right)=O(\ell^2\log \ell).
\end{gather}
\end{remark}

\section*{Acknowledgments}

I am very grateful to Alexei Borodin (who, by the way, suggested
this problem), Christophe Garban,  Beno\^it Laslier and Herbert Spohn for
many enlightening comments and to Giada Basile, Lorenzo Bertini and
Stefano Olla for  discussions on the  ``gradient condition''. 

\end{document}